\documentclass[a4paper]{article}
\usepackage[utf8]{inputenc}

\usepackage[margin=1in]{geometry}

\usepackage{amsmath, amssymb, listings}
\usepackage[hidelinks]{hyperref}

 \usepackage{xcolor}

\usepackage{amsthm}
\usepackage{aliascnt}
\usepackage[noabbrev]{cleveref}

\newtheorem{thm}{Theorem}[section]

\newcommand{\newaliasthm}[2]{%
  \newaliascnt{#1}{thm}%
  \newtheorem{#1}[#1]{#2}%
  \aliascntresetthe{#1}%
  \crefname{#1}{#2}{#2s}%
}

\newaliasthm{lm}{Lemma}
\newaliasthm{res}{Result}
\newaliasthm{crl}{Corollary}
\newaliasthm{prop}{Proposition}
\newaliasthm{obs}{Observation}
\newaliasthm{rmk}{Remark}

\newaliasthm{df}{Definition}
\newaliasthm{nt}{Notation}
\newaliasthm{ex}{Example}
\newaliasthm{prob}{Problem}

\crefname{thm}{Theorem}{Theorems}

\newcommand{\eps}{\varepsilon}
\renewcommand{\phi}{\varphi}

\newcommand{\RR}{\mathbb R}
\newcommand{\FF}{\mathbb F}
\newcommand{\CC}{\mathbb C}

\newcommand{\NN}{\mathbb N}

\newcommand{\vspan}[1]{\left \langle #1 \right \rangle}
\newcommand{\vspann}[2]{\left \langle #1 \, \| \, #2 \right \rangle}
\newcommand{\set}[1]{ \left \{ #1 \right \} }
\newcommand{\sett}[2]{ \left\{ #1 \, \, || \, \, #2 \right \} }

\newcommand{\one}{\mathbf 1}
\newcommand{\zero}{\mathbf 0}
\newcommand{\floor}[1]{\left \lfloor #1 \right \rfloor}
\newcommand{\ceil}[1]{\left \lceil #1 \right \rceil}

\newcommand{\ma}{\mathcal A}
\newcommand{\mb}{\mathcal B}

\newcommand{\me}{\mathcal E}
\newcommand{\mf}{\mathcal F}

\newcommand{\mi}{\textnormal I}

\newcommand{\mm}{\mathcal M}
\newcommand{\mn}{\mathcal N}
\newcommand{\mo}{\mathcal O}
\renewcommand{\mp}{\mathcal P}
\newcommand{\mq}{\mathcal Q}

\newcommand{\ms}{\mathcal S}

\newcommand{\fs}{\mathfrak S}
\newcommand{\ft}{\mathfrak T}

\newcommand{\BP}{\boldsymbol P}
\newcommand{\BQ}{\boldsymbol Q}
\newcommand{\fq}{\FF_q}
\renewcommand{\th}{\textsuperscript{th} }

 \DeclareMathOperator{\supp}{supp}
 \DeclareMathOperator{\wt}{wt}

 \DeclareMathOperator{\PG}{PG}
    \newcommand{\pg}{\PG}

 \DeclareMathOperator{\PGL}{PGL}
    \newcommand{\pgl}{\PGL}

 \DeclareMathOperator{\PGamL}{P\Gamma L}
 \DeclareMathOperator{\GL}{GL}

 \DeclareMathOperator{\proj}{proj}
 \DeclareMathOperator{\Cay}{Cay}
 \renewcommand{\hom}[1]{\FF_q[X,Y]_{#1}}

 \renewcommand{\leq}{\leqslant}
 \renewcommand{\geq}{\geqslant}
 \newcommand{\fqi}{\fq \cup \set \infty}

\title{Intersection problems for linear codes and polynomials over finite fields}
\author{Sam Adriaensen\thanks{Department of Mathematics and Data Science, Vrije Universiteit Brussel, Pleinlaan 2, 1050 Elsene, Belgium. \href{mailto:sam.adriaensen@vub.be}{sam.adriaensen@vub.be}.
The author is supported by a postdoctoral fellowship 12A3Y25N from the Research Foundation Flanders (FWO).}}
\date{}

\begin{document}

\maketitle

\begin{abstract}
 This paper proves a stability result for a variation of the Erd\H os-Ko-Rado theorem in the context of polynomials over finite fields.
 Let $\mathcal F$ be a family of polynomials of degree at most $k \geq 3$ in $\mathbb F_q[X]$.
 Call $\mathcal F$ \emph{intersecting} if for any two polynomials $f, g$ in $\mathcal F$, there exists a point $x \in \mathbb F_q$ for which $f(x) = g(x)$.
 An intersecting family is called a \emph{star} if it consists of all polynomials $f$ with $\deg f \leq k$ such that $f(x) = y$ for some fixed points $x, y \in \mathbb F_q$.
 In this paper we prove that if $\mathcal F$ is an intersecting family with $|\mathcal F| \geq \frac 1{\sqrt 2} q^k + \mathcal O(q^{k-1})$, then $\mathcal F$ is contained in a star.
 In fact, we prove that this is still true if we also evaluate the polynomials ``at infinity'', which is equivalent to studying the problem for homogeneous bivariate polynomials.

The proof technique extends to a general framework for intersection problems of linear codes $C$.
One has to investigate the geometry of the projective system $\mathcal S$ associated to $C$.
If the hyperplanes that don't intersect $\mathcal S$ are well spread out with respect to the points not on $\mathcal S$, then one obtains  stability results, showing that any intersecting family of reasonably large size is contained in a star.
\end{abstract}

\paragraph{Keywords:} Intersection problems, Erd\H os-Ko-Rado theorem, Finite geometry, Linear codes, Projective systems, Normal rational curves, Polynomials, Spectral graph theory.

\paragraph{MSC:}
05D05, 
11T06, 
51E22, 
05C50, 
05E30. 

\section{Introduction}

The main contribution of this paper is \Cref{Thm:Main}, which proves a stability theorem for intersection problems concerning polynomials over finite fields.
In the first part of the Introduction, we introduce a framework for intersection problems.
In the second part, we summarise the known results concerning intersection problems for polynomials over finite fields, and present our main result.
The range of the proof strategy extends beyond polynomials over finite fields.
We can represent the problem as dealing with linear codes.
In the third part of the Introduction, we explain the framework of intersection problems for linear codes, and present helpful tools to characterise large intersecting families.
\Cref{Thm:Main} will eventually be proved by an application of these linear code-related tools.
In the fourth part, we give an overview of the proof strategy and the paper.

\subsection{General intersection problems}

In their seminal paper, Erd\H os, Ko and Rado \cite{EKR} proved that if $\mf$ is a family of subsets of size $k$ of an $n$-element set $\mp$, with $n > 2k$, then $|\mf| \leq \binom{n-1}{k-1}$, with equality if and only $\mf$ consists of all $k$-element subsets through a fixed point of $\mp$.
This consequently inspired its own branch in combinatorics, which we could denominate \emph{intersection problems}.
Much research is still being done.
See e.g.\ \cite{Ellis, Kupavskii:Zakharov, Tanaka:Tokushige, Trees, Bulavka:Woodroofe, Salia:Toth} for some recent progress.

In this paper, we will study intersection problems in the following framework.
We start with an \emph{incidence structure} $(\mp,\mb,\mi)$.
This means that $\mp$ and $\mb$ are sets (in our case always finite), and $\mi \subseteq \mp \times \mb$ is a relation.
We call the elements of $\mp$ \emph{points}, the elements of $\mb$ \emph{blocks}, and we call $\mi$ the \emph{incidence relation}.
We write $P \mi B$ as a shorthand for $(P,B) \in \mi$, and we say that $P$ is on $B$, or $B$ goes through $P$, or $P$ is incident with $B$.

We call two blocks $B_1$ and $B_2$ \emph{$t$-intersecting} if there are at least $t$ points incident with both $B_1$ and $B_2$.
A family $\mf \subseteq \mb$ of blocks is said to be \emph{$t$-intersecting} if any two blocks contained in $\mf$ are $t$-intersecting.
If $t=1$, we simply say \emph{intersecting}.
A \emph{$t$-star} (also called a  \emph{canonical} or \emph{trivially $t$-intersecting family} by other authors) is a non-empty $t$-intersecting family of the form $\sett{B \in \mb}{P_1, \dots, P_t \mi B}$,
with $P_1, \dots, P_t$ distinct points.
A $1$-star is simply called a \emph{star}.
The question arises whether the largest $t$-intersecting families are $t$-stars, and if so, how large can a $t$-intersecting family be that is not contained in a $t$-star?

\subsection{Intersection problems for polynomials}

We will investigate intersecting families in the context of polynomials over finite fields of bounded degree.
Let $q$ be a prime power, $\fq$ the finite field of order $q$, let $k < q$ be an integer, and let $\fq[X]_{\leq k}$ denote the set of polynomials in $\fq[X]$ of degree at most $k$.

One possible definition for two polynomials $f$ and $g$ to be $t$-intersecting is that $\deg(\gcd(f,g)) \geq t$.
This approach was taken in \cite{Salia:Toth}, but we will not expand on it here.

Instead, we say that polynomials $f$ and $g$ are $t$-intersecting if $f(X)=g(X)$ has at least $t$ solutions in $\fq$.
Note that this implies that if $f \neq g$, then $\deg(f-g) \geq t$.
We easily obtain the following result.

\begin{res}[{\cite[Lemma 6.1]{A-3-EKR}}]
 \label{Res:Weak EKR}
 Suppose that $t \leq k < q$.
 A $t$-intersecting family in $\fq[X]_{\leq k}$ contains at most $q^{k+1-t}$ polynomials.
\end{res}

Note that the above bound is attained by the $t$-stars.
Moreover, some stability results are known for large intersecting families in $\fq[X]_{\leq k}$.

\begin{res}
 \label{Res:Uni Pol}
 Suppose that $\mf$ is an intersecting family in $\fq[X]_{\leq k}$ with $2 \leq k < q$.
 \begin{enumerate}
  \item \label{Res:Uni Pol:1} \textnormal{\cite{A-3-EKR}} Then $|\mf| \leq q^k$, with equality if and only if $\mf$ is a star.
  \item \label{Res:Uni Pol:2} \textnormal{\cite{A3-2022-Stability}} If $k=2$, and $|\mf| \geq \frac 1 {\sqrt 2} q^2 + 2 \sqrt 2 q + 8$, then $\mf$ is contained in a star.
 \item \label{Res:Uni Pol:3} \textnormal{\cite{Aguglia:Csajbok:Weiner}} Suppose that $q \geq 8$ if $q$ is even, and $q \geq 53$ if $q$ is odd. If $|\mf| \geq q^k - q^{k-1}$, then $\mf$ is contained in a star.
 \end{enumerate}
\end{res}

The proofs of \Cref{Res:Uni Pol} (\ref{Res:Uni Pol:1}, \ref{Res:Uni Pol:3}) rely on proving the result for the case $k=2$, and then applying induction.
However, the problem becomes more intricate if we also evaluate the polynomials ``at infinity''.

\begin{df}
 \label{Df:infty}
We define for a polynomial $f(X) = \sum_{i=0}^k a_i X^i$ of degree at most $k$ that $f(\infty_k) = a_k$.
If $k$ is clear from context, we denote $\infty_k$ as $\infty$.
\end{df}

One can think of evaluating $f$ on $\fqi$ as evaluating $f$ on the projective line $\pg(1,q)$.
Indeed, if we define $F(X,Y) = Y^k f(X/Y) = \sum_{i=0}^k a_i X^i Y^{k-i}$, then $F(X,1) = f(X)$ and $F(1,0) = f(\infty)$.
Hence, two polynomials $f,g \in \fq[X]_{\leq k}$ are $t$-intersecting on $\fqi$ if and only if there are $t$ points $\vspan{(x_1,y_1)}, \dots, \vspan{(x_t,y_t)}$ of $\pg(1,q)$ where $Y^k f(X/Y)$ and $Y^k g(X/Y)$ agree.

We denote by $\hom k$ the set of homogeneous polynomials in $\fq[X,Y]$ of degree $k$ (including the zero polynomial), and make the following convention.

\begin{df}
 \label{Df:Pol int}
\begin{enumerate}
 \item When we say that $f,g \in \fq[X]_{\leq k}$ are \emph{$t$-intersecting} if $f(X) = g(X)$ has at least $t$ solutions in $\fq$.
 This corresponds to the incidence structure $(\fq^2, \fq[X]_{\leq k}, \mi)$, with $\mi$ the natural incidence relation:
 \[
  \mi = \sett{((x,y),f) \in \fq^2 \times \fq[X]_{\leq k}}{f(x) = y}.
 \]
 \item We say that $F, G \in \hom k$ are \emph{$t$-intersecting} if $F(X,Y) = G(X,Y)$ has at least $t$ pairwise linearly independent solutions on $\fq^2 \setminus \set \zero$.
 Equivalently, if $f(X) = F(X,1)$ and $g(X) = G(X,1)$, then $f(X) = g(X)$ has at least $t$ solutions on $\fqi$.
 Hence, this corresponds to the incidence structure $((\fq \cup \{\infty_k\})\times \fq, \fq[X]_{\leq k}, \mi)$, with natural incidence.
\end{enumerate}
\end{df}
A $t$-intersecting family in $\fq[X]_{\leq k}$ can be interpreted as a $t$-intersecting family in $\hom k$, but not vice versa.
We can thus carry characterisation results in $\hom k$ over to $\fq[X]_{\leq k}$.

The crucial difference between the incidence structures is that for any $k' < k$, $(\fq^2, \fq[X]_{\leq k},\mi)$ has $(\fq^2, \fq[X]_{\leq k'},\mi)$ as a substructure.
This no longer holds for $\hom k$, since altering $k$ means that we alter $f(\infty)$.
We can no longer apply induction.

\bigskip

Not much is known in the homogeneous setting.
It is easy to extend \Cref{Res:Weak EKR} to this setting if $t < k$.
However, if $t = k$, we find a quadratic upper bound on $k$-intersecting families in $\hom k$, while we would expect a linear upper bound.

\begin{prop}
 \label{Prop:t int}
 Suppose that $t \leq k < q$, and $k \geq 2$.  Let $\mf$ be a $t$-intersecting family in $\hom k$. \begin{enumerate}  \item If $t < k$, then $|\mf| \leq q^{k+1-t}$.  \item If $t = k$, then $|\mf| < \frac{q^2-1}k + 1$. \end{enumerate}
\end{prop}

We note that \Cref{Prop:t int} (2) extends \cite[Proposition 4.6.2]{thesis}, which in turn uses an argument by Blokhuis and Bruen \cite{Blokhuis:Bruen}.
Moreover, there is a good reason why the case $t=k$ does not give the upper bound $|\mf| \leq q$.
Namely, if $q$ is odd, then there are $2$-intersecting families in $\hom 2$ of size $\frac{3q-1}2$ \cite[Proposition 4.6.1]{thesis}.
This construction is based on \cite{Blokhuis:Bruen}.

\bigskip

Let us now focus on the case $t=1$.
A similar stability result to \Cref{Res:Uni Pol} (\ref{Res:Uni Pol:2}) holds.

\begin{res}[{\cite[Corollary 6.7]{A3-2022-Stability}}]
 \label{Res:Hom Pol}
 If $\mf$ is an intersecting family in $\hom 2$ with $|\mf| \geq \frac 1 {\sqrt 2} q^2 + 2 \sqrt 2 q + 8$, then either $\mf$ is contained in a star, or $q$ is even, and $\mf$ is contained in an intersecting family of the form $\sett{a X^2 + b XY + c Y^2}{a,c \in \FF_q}$ for some fixed $b \in \FF_q$.
\end{res}

In this paper, we extend the result to polynomials of higher degree.
We will prove the following theorem, which is the main theorem of this paper.

\begin{thm}
 \label{Thm:Main}
 For every integer $k \geq 3$, there exists a constant $\sigma_k$,  such that for all prime powers $q$, every intersecting family $\mf \subseteq \hom k$ with $|\mf| \geq \frac1{\sqrt 2} q^k + \sigma_k q^{k-1}$ is contained in a star.
\end{thm}

Note that since an intersecting family in $\fq[X]_{\leq k}$ represents an intersecting family in $\hom k$, \Cref{Thm:Main} also holds for $\fq[X]_{\leq k}$.
We remark that we do not believe the constant $\frac 1 {\sqrt 2}$ to be optimal, but it is the optimum of the current proof techniques.

\subsection{Intersection problems for linear codes}

\Cref{Thm:Main} can be interpreted to be dealing with the intersection problem for \emph{extended Reed-Solomon codes}.
More generally, we can discuss intersection problems for linear codes.
Given a finite field $\fq$, a $k$-dimensional subspace $C$ of $\fq^n$ is called a \emph{linear code} with parameters $[n,k]_q$, or more briefly a linear $[n,k]_q$ code.
Throughout the paper, we always assume that the linear codes under consideration are \emph{projective}, which means that there do \underline{not} exist coordinate positions $i$ and $j$, and a scalar $\alpha \in \fq$, such that $c(i) = \alpha c(j)$ for all $c \in C$.

\begin{df}
 \label{Df:Code int}
 Given a linear code $C$, we can define the incidence structure $(\set{1, \dots, n} \times \fq, C, \mi)$, where $(i,\alpha)$ is incident with $c \in C$ if and only if $c(i) = \alpha$.
 Then $\mf \subseteq C$ is a \emph{$t$-intersecting family} if and only if for any two codewords $c, c' \in C$, there exist at least $t$ coordinate positions $i$ where $c(i) = c'(i)$.
 Equivalently, no two elements of $\mf$ are at Hamming distance more than $n-t$.
\end{df}

Let us introduce some general terminology, but before doing so introduce a useful convention.
Throughout the paper, we will work with real vector spaces whose coordinate positions are labelled by a finite set $X$.
It is then easier to interpret this vector space as the set $\RR^X$ of functions $X \to \RR$.
These functions will still be referred to as vectors.
Similarly, if $M$ is a matrix whose rows and columns are labelled by finite sets $X$ and $Y$ respectively, we will represent it as a function $M: X \times Y \to \RR$.
The same convention holds for complex vector spaces.

\begin{df}
 Given a set $S$, define its \emph{characteristic vector} $\chi_S$ to be the function that takes the value 1 on the elements of $S$ and the value 0 anywhere else.
\end{df}
Typically when we discuss the characteristic vector of $S$, we restrict its domain to a natural superset of $S$.
In particular, for a family $\mf \subseteq \mb$ of blocks, we let the domain of $\chi_\mf$ be $\mb$, so that $\chi_\mf \in \RR^\mb$.

Based on the terminology from \cite{Maegher}, we say that $(\mp,\mb,\mi)$ has the
\begin{itemize}
 \item \emph{weak EKR property} if the size of the largest star is the size of the largest intersecting family (typically we study settings where all stars have the same size),
 \item \emph{EKR module property} if the characteristic vector of any intersecting family of maximum size is a linear combination of characteristic vectors of stars,
 \item \emph{strict EKR property} if the only intersecting families of maximum size are stars,
 \item \emph{HM property} (named after Hilton and Milner \cite{Hilton:Milner}) if the strict EKR property holds, and the largest intersecting families that are not contained in a star are of the following form for some point $P \in \mp$, and some block $B \in \mb$ that is not incident with $P$:
 \[
  \mf = \sett{B' \in \mb}{P \mi B', \, (\exists Q \mi B')(Q \mi B)} \cup \set B.
 \]
 In other words, $\mf$ consists of the block $B$, and all blocks through $P$ that intersect $B$.
 This family is easily checked to be intersecting, and any intersecting family of this form is said to be of \emph{Hilton-Milner type}.
\end{itemize}

We will investigate for which linear codes $C$ the above properties hold.
This is intimately tied to the geometry of the \emph{projective system} of $C$ (see \Cref{Sec:Codes} for the definition).
We will prove the following theorem.

\begin{thm}
 \label{Thm:Codes}
Let $C$ be a linear $[n,k]_q$ code with projective system $\ms$ in $\pg(k-1,q)$.
Let $\mm$ denote the set of hyperplanes that don't intersect $\ms$.
\begin{enumerate}
 \item \label{Thm:Codes:1} $C$ has the weak EKR property if and only if $\mm$ is not empty.
 \item \label{Thm:Codes:2} $C$ has the EKR module property if and only if every point $P \notin \ms$ is incident with at least one hyperplane of $\mm$.
 \item \label{Thm:Codes:3} If no 3 points of $\ms$ are collinear, and for every point $P \notin \ms$,
  \[
   \left| |\sett{\Pi \in \mm}{P \in \Pi}| - \frac{|\mm|}q  \right| < \frac{|\mm|}{q \min \{ \sqrt n, q-1 \}},
  \]
  then $C$ has the strict EKR property.
\end{enumerate}
\end{thm}

In addition, in order to prove \Cref{Thm:Main}, we provide a set of sufficient conditions for a family of linear codes to ensure a stability result, putting us one step closer towards establishing the HM property.

\begin{thm}
 \label{Thm:Stability}
 Let $\mq$ be an infinite set of prime powers, and suppose that for each $q \in \mq$, there exists a linear code $C_q$ with parameters $[n_q,k]_q$ (where $n_q$ can vary, but $k$ is fixed).  Denote by $\ms_q$ the projective system associated with $C_q$, and by $\mm_q$ the set of hyperplanes that don't intersect $\ms_q$.  Suppose that there exist constants $a,b,\delta,\mu,\tau$ with $0 < \mu < 1$ such that the following properties hold for all of $q \in \mq$:  \begin{enumerate}   \item no 3 points of $\ms_q$ are collinear,   \item $n_q \leq aq+b$,   \item $\big| |\mm_q| - \mu q^{k-1}\big| \leq \tau q^{k-2}$,   \item for each point $P$ of $\pg(k-1,q)$ outside of $\ms_q$, $\left| \sett{\Pi \in \mm_q}{P \in \Pi} - \frac{|\mm_q|}q \right| \leq \delta q^{k-3}$.  \end{enumerate}  Then there exists a constant $\sigma$ such that in each of the codes $C_q$, every intersecting family $\mf$ with \[  |\mf| \geq \max\left\{ \frac1{\sqrt{1+a^{-1}}}, \, 1-\mu \right\} q^{k-1} + \sigma q^{k-2}  \] is contained in a star.
\end{thm}

\subsection{Overview of the proof strategy and the paper}
 \label{Sec:Overview}

A crucial tool in the proof strategy are certain Cayley graphs, sometimes called \emph{linear graphs}, see e.g.\ \cite[\S 6.4]{Godsil}.
These are Cayley graphs $\Gamma$ over a finite vector space $\fq^k$, where the connection set $S$ is closed with respect to non-zero scalar multiplication.
Associate with $S$ the set $\mm$ of hyperplanes in $\pg(k-1,q)$ which are of the form $v^\perp$ for some $v \in S$.
Then the eigenvalues of $\Gamma$ are determined by the distribution of the hyperplanes of $\mm$ over the points of $\pg(k-1,q)$.
In the context of a linear code $C$, we take $\mm$ to be the set of hyperplanes that don't intersect the projective system associated with $C$.
The corresponding Cayley graph $\Gamma_0(C)$ is then the Cayley graph on $C$ where adjacency is given by not being intersecting.
Note that cocliques in this graph coincide with intersecting families.

Using this geometric framework for studying eigenvalues of linear graphs, we can prove  \Cref{Thm:Codes} (\ref{Thm:Codes:1}, \ref{Thm:Codes:2}).
This will be the main focus of \Cref{Sec:Mod prop}

The proofs of \Cref{Thm:Codes} (\ref{Thm:Codes:3}) and \Cref{Thm:Stability} require a more advanced strategy.
We build on the tools developed by the author in \cite{A3-2022-Stability}, and vastly expand their applicability.
The strategy is as follows.
Consider a linear code $C$, a large intersecting family $\mf$ in $C$, and let $\mf_{i,\alpha}$ denote the star $\sett{c \in C}{c(i) = \alpha}$.
\begin{enumerate}
 \item \label{Step 1} We prove that for any star $\mf_{i,\alpha}$, $|\mf \cap \mf_{i,\alpha}|$ is either small or large, i.e.\ we prove that $|\mf \cap \mf_{i,\alpha}| \notin [s,t]$ for some integers $s < t$.
 In order to prove this, we construct a graph $B(C,i,\alpha)$.
 This graph is bipartite, with one part being the star $\mf_{i, \alpha}$, the other part the remainder of $C$, and adjacency is given by not being intersecting.
 Then $\mf$ is still a coclique in this graph, and the bipartite expander mixing lemma gives a bound on $|\mf \cap \mf_{i,\alpha}| \cdot |\mf \setminus \mf_{i,\alpha}|$.
 This will be done in \Cref{Sec:Few or many}.
 \item \label{Step 2} We then prove that there is a star $\mf_{i,\alpha}$ that contains a sizeable portion of $\mf$.
 To start, we need a good lower bound on the second smallest eigenvalue of $\Gamma_0(C)$.
 We can then use this to prove that the characteristic vector $\chi_\mf$ of $\mf$, when orthogonally projected onto the space spanned by the characteristic vectors of the stars, is large.
 By an averaging argument, it follows that $\mf$ has a large intersection with some star $\mf_{i,\alpha}$.
 This is done in \Cref{Sec:More than few}.
 \item \label{Step 3} To finish the proof, we use Step (\ref{Step 2}) to prove that some star $\mf_{i,\alpha}$ intersects $\mf$ in at least $s$ points (with $s$ as in Step (\ref{Step 1}).
 By Step \ref{Step 1}, this implies that $\mf_{i,\alpha}$ intersects $\mf$ in more than $t$ points.
 If $t$ is large enough, this suffices to conclude that $\mf$ is a subset of $\mf_{i,\alpha}$.
\end{enumerate}
We use this strategy to prove \Cref{Thm:Codes} (\ref{Thm:Codes:3}) in \Cref{Sec:Strict sub} and \Cref{Thm:Stability} in \Cref{Sec:Stab}.

In \Cref{Sec:Stab pol}, we first prove \Cref{Prop:t int}, and then apply \Cref{Thm:Stability} to extended Reed-Solomon codes to prove \Cref{Thm:Main}.

We conclude the paper with some open problems in \Cref{Sec:Conclusion}.

In \Cref{Sec:Assoc Sch}, we delve into the theory of association (and specifically translation) schemes.
This allows us to determine the exact eigenvalues of some graphs defined on homogeneous polynomials over finite fields.
These eigenvalues can be useful for further investigation of intersection problems.

\section{Preliminaries}
 \label{Sec:Prel}

\subsection{Projective geometry}

We start by recalling some basic notions from finite projective geometry.
We refer the reader to \cite{Hirschfeld:79}.
The projective space $\pg(k-1,q)$ is defined to be the collection of all vector subspaces of $\FF_q^k$, equipped with the natural incidence relation, i.e.\ two elements $\pi$ and $\rho$ of $\pg(k-1,q)$ are incident if and only if one is contained in the other.
We will use terminology from projective geometry.
In particular, if $U$ is a subspace of $\FF_q^k$ and its bases have $t$ elements, then we say that it has (projective) dimension $t-1$.
Subspaces of dimension $0$, $1$, and $2$ are called points, lines, and planes respectively.
Subspaces of codimension $1$ are called hyperplanes.
Given a non-zero vector $(x_1, \dots, x_k)^\top$ of $\FF_q^k$, the projective point $\vspan{(x_1, \dots, x_k)^\top}$ is simply denoted as $(x_1 : \ldots : x_k)$.

\begin{df}
 \label{Df:Projection}
 Consider a subspace $\rho$ in $\pg(k-1,q)$ of dimension $r-1$.
 We define the \emph{quotient space} through $\rho$ as the collection of subspaces that contain $\rho$, or equivalently (since $\rho$ is a subspace of $\fq^k$), as the projective space associated to the vector space $\fq^k/\rho$.
 The quotient space is then isomorphic to $\pg(k-r-1,q)$.
 
 We define the \emph{projection} from $\rho$, denoted $\proj_\rho$, in $\pg(k-1,q)$ as the map that sends a subspace $\sigma$ of $\pg(k-1,q)$ to the subspace $\vspan{\rho,\sigma}/\rho$ in the quotient space through $\rho$.
 We can think of this geometrically as follows.
 Take a subspace $\Pi$ of $\pg(k-1,q)$ disjoint to $\rho$ of maximal dimension, i.e.\ $\dim \Pi = k-r-1$.
 Then we can think of $\proj_\rho$ as the map sending $\sigma$ to $\vspan{\sigma,\rho} \cap \Pi$.
\end{df}

\subsection{Eigenvalue bounds}

Eigenvalue bounds from spectral graph theory form a versatile framework to tackle intersection problems.
Suppose that $R \subseteq X \times X$ is a symmetric binary relation on a finite set $X$.
We can interpret $R$ as an undirected graph $\Gamma$ (possibly containing loops) on $X$.
We define the \emph{adjacency matrix} of $R$ (or of $\Gamma$) as
\[
 A_R = A_\Gamma: X \times X \to \RR: (x,y) \mapsto \begin{cases} 1 & \text{if } (x,y) \in R, \\ 0 & \text{otherwise}. \end{cases}
\]
Since $R$ is symmetric, $A_R$ is a symmetric matrix, and thus admits an orthonormal basis of eigenvectors, and has only real eigenvalues.
We let $\lambda_1 \geq \dots \geq \lambda_{|X|}$ denote the eigenvalues of $A_R$, counted by multiplicity.
If we want to emphasize the relation $R$ or the graph $\Gamma$, we might write $\lambda_i(R)$ or $\lambda_i(\Gamma)$ instead of $\lambda_i$.

The \emph{expander mixing lemma} provides a ubiquitous framework for graph eigenvalue bounds.
We state two such bounds, which play an important role in this paper.
First some definitions and conventions:
\begin{itemize}
 \item A \emph{coclique} in a graph is a set of vertices which are pairwise not adjacent.
 A \emph{clique} is a set of vertices which are pairwise adjacent.
 \item If $S$ is a subset of vertices of $\Gamma$, the domain of the characteristic vector $\chi_S$ of $S$ is understood to be the vertex set of $\Gamma$.
 \item We denote the all-one vector (or the constant function taking value 1) by $\one$, and the zero vector by $\zero$.
\end{itemize}

\begin{res}[Hoffman ratio bound {\cite[Theorem 2.4.1]{Godsil:Meagher}}]
 Let $\Gamma$ be a $k$-regular graph on $n$ vertices.
 Suppose that $S$ is a coclique in $\Gamma$.
 Then
 \[
  |S| \leq \frac{n}{\frac{k}{|\lambda_n|}+1}.
 \]
 Equality holds if and only if $\chi_S$ is in the span of $\one$ and the $\lambda_n$-eigenspace of $A_\Gamma$.
\end{res}

The Hoffman ratio bound has proved itself to be a versatile tool to establish the EKR module property in many different settings.
We refer the reader to \cite{Godsil:Meagher} for a detailed account of this approach.

\bigskip

The second eigenvalue bound that we discuss, concerns bipartite graphs.
Recall that a graph $\Gamma$ is called \emph{bipartite} if its vertices can be partitioned into two cocliques $L$ and $R$.
We call $\Gamma$ \emph{$(d_L,d_R)$-biregular} if the degree of any vertex in $L$ is $d_L$ and the degree of any vertex in $R$ is $d_R$.
Given a pair of sets $S \subseteq L$ and $T \subseteq R$ of vertices, we denote by $e(S,T)$ the number of edges in $\Gamma$ with endpoints in $S$ and $T$.

\begin{res}[Bipartite expander mixing lemma {\cite[Theorem 5.1]{Haemers}}]
 \label{Res:Bip EML}
 Let $\Gamma$ be a $(d_L,d_R)$-biregular bipartite graph on $L \cup R$.
 Take $S \subseteq L$ and $T \subseteq R$.
 Then
 \[
  \left( e(S,T) - \frac{d_R}{|L|} |S| |T| \right)^2 \leq \lambda_2^2 |S| \frac{|L|-|S|}{|L|} |T| \frac{|R|-|T|}{|T|}.
 \]
\end{res}

\subsection{Eigenvalues of abelian Cayley (multi)graphs}
 \label{Sec:Eigval Cayley}

Suppose that $G$ is an abelian group.
Let $M$ be a multiset with $G$ as ground set.
Formally, this means that $M$ is defined by a multiplicity function $\mu: G \to \NN$.
Then $M$ defines a directed multigraph $\Gamma$ with vertex set $G$, where the arc $(g,h)$ has multiplicity $\mu(h-g)$.
We call $\Gamma$ a directed \emph{Cayley multigraph}, and denote it as $\Gamma = \Cay(G,M)$ or $\Cay(G,\mu)$.
Note that $\Cay(G,M)$ is a simple, loopless, undirected graph if and only if $M$ is an ordinary set that is closed under taking inverses, and $M$ does not contain the identity element.
In this case, we call $\Cay(G,M)$ a \emph{Cayley graph}.

The adjacency matrix of a directed Cayley multigraph $\Gamma = \Cay(G,M)$ is defined as the matrix
\[
 A_\Gamma: G \times G \to \NN: (g,h) \mapsto \mu(g-h).
\]
The characters of $G$ constitute an orthogonal basis of eigenvectors of $A_\Gamma$, see e.g.\ \cite[Theorem 1]{Liu=Zhou}.
Indeed, if $\chi:G \to \CC$ is a character of $G$, then
\[
 (A_\Gamma \chi)(g) = \sum_{h \in G} A_\Gamma(g,h) \chi(h) = \sum_{h \in G} \mu(h-g) \chi(h) = \sum_{h' \in G} \mu(h') \chi(g+h') = \left( \sum_{h' \in G} \mu(h') \chi(h') \right) \chi(g).
\]
It follows that $\chi$ is an eigenvector of $A_\Gamma$ with eigenvalue $\sum_{h \in G} \mu(h) \chi(h)$, which we denote by $\chi(\Gamma)$.
It is also well known that the group of characters $\hat G$ is isomorphic to $G$, and there is a way to write $\hat G = \sett{\chi_g}{g \in G}$ such that $G \to \hat G: g \mapsto \chi_g$ is a group isomorphism.

In case $G = \FF_q^k$ with $q = p^h$ and $p$ prime, then the characters can be defined as follows.
Let $T:\FF_q\to \FF_p$ denote a non-zero $\FF_p$-linear functional, and let $\zeta \in \CC$ be a primitive $p$\th root of unity.
Then we can take $\chi_v: \FF_q^k \to \CC: w \mapsto \zeta^{T(v^\top w)}$.
Now consider the special case where $M$ is invariant under non-zero scalar multiplication, i.e.\ $\mu(\alpha v) = \mu(v)$ for all non-zero scalars $\alpha \in \FF_q^*$ and vectors $v \in \FF_q^k$.
Let $G'$ be a subset of $\FF_q^k$ that contains a unique scalar multiple of each non-zero vector of $\FF_q^k$, i.e.\ a unique vector representative for each point of $\pg(k-1,q)$.
Then
\begin{equation}
 \label{Eq:Characters}
 \chi_v(\Gamma) = \mu(\zero) + \sum_{w \in G'} \sum_{\alpha \in \FF_q^*} \mu(\alpha w) \chi_v(\alpha w) = \mu(\zero) + \sum_{w \in G'} \mu(w) \sum_{\alpha \in \FF_q^*} \zeta^{T(\alpha v^\top w)}.
\end{equation}
If we consider $\sum_{\alpha \in \FF_q} \zeta^{T(\alpha v^\top w)}$, then either $v^\top w = 0$, and every term in the sum is one, or $v^\top w \neq 0$,  and $\alpha v^\top w$ takes every value in $\FF_q$ exactly once.
In that case, the sum sees every power of $\zeta$ exactly $q/p$ times, which makes the sum equal to zero.
Considering that (\ref{Eq:Characters}) omits the term for $\alpha = 0$, we find that
\begin{equation}
 \label{Eq:Characters 2}
 \chi_v(\Gamma) =  \mu(\zero) + (q-1)\sum_{w \in G' \cap v^\perp} \mu(w) - \sum_{w \in G' \setminus v^\perp} \mu(w)
 = \mu(\zero) + q \sum_{w \in G' \cap v^\perp} \mu(w) - \sum_{w \in G'} \mu(w).
\end{equation}
If $G$ is a Cayley graph, then $M$ is an ordinary set, not containing $\zero$.
Define the set $M' = G' \cap M$ containing a unique scalar multiple of each element of $M$.
Then the above equation simplifies to
\begin{align}
 \label{Eq:Characters 3}
 \chi_v(\Gamma) = q |M' \cap v^\perp| - |M'|.
\end{align}

\subsection{Linear codes and projective systems}
 \label{Sec:Codes}

We recall some notions from coding theory.
The reader is referred to \cite{Ball} or \cite{MacWilliams:Sloane} for more background.
Let $q$ denote a prime power, and $\FF_q$ the finite field of order $q$.
A \emph{linear $[n,k]_q$-code} is a $k$-dimensional subspace $C$ of $\FF_q^n$.
Recall that $C$ is called \emph{projective} if no two coordinates of $C$ are linearly dependent, i.e.\ for every $i$ and $j$, $\sett{(c(i),c(j))}{c \in C} = \fq^2$. 
We assume all codes to be projective.
If $C' \leq C$ is a linear $[n,k']_q$ code, we call $C'$ a \emph{subcode} of $C$.
The \emph{support} of a vector $v \in \FF_q^n$ is the set
\[
 \supp(v) = \sett{i \in \{1, \dots, n\}}{ v(i) \neq 0}
\]
of coordinate positions where $v$ has a non-zero entry, and the \emph{(Hamming) weight} of $v$ is defined as $\wt(v) = |\supp(v)|$, i.e.\ the number of non-zero entries in $v$.
The \emph{minimum weight} of a linear code $C$ is defined as
\[
 d = \min \sett{\wt(c)}{c \in C \setminus \set \zero}.
\]
If we want to emphasize the minimum weight $d$ of a linear $[n,k]_q$ code $C$, we may write that $C$ is a linear $[n,k,d]_q$ code.
The \emph{Singleton bound} states that if a linear $[n,k,d]_q$ code exists, then $d \leq n-k+1$.
In case of equality, we say that $C$ is an \emph{MDS code}, which stands for maximum distance separable.
The \emph{weight distribution} of $C$ is defined as the sequence $(W_0, \dots, W_n)$, where $W_t = |\sett{c \in C}{\wt(c) = t}|$ equals the number of codewords of weight $t$.
The weight distribution of MDS codes is known.

\begin{res}[{\cite[Theorem 11.6 (page 321)]{MacWilliams:Sloane}}]
 \label{Res:Wt Dist MDS}
 Suppose that $C$ is a linear $[n,k,d = n-k+1]_q$ MDS code.
 Then its weight distribution is given by $(1,0,\dots,0,W_d, \dots, W_n)$ where for $t \geq d$,
 \[
  W_t = (q-1) \binom {n}t \sum_{j=0}^{t-d} (-1)^j \binom {t-1}j q^{t-d-j}.
 \]
\end{res}

The prototypical example of MDS codes is given by the extended Reed-Solomon codes.
Recall that if we fix a degree $k$, then for all polynomials $f(X) = \sum_{i=0}^k a_i X^i$ of degree at most $k$, we define $f(\infty) = a_k$.

\begin{df}
 \label{Df:RS}
 Consider the finite field $\FF_q$, and choose an integer $k < q$.
 Label the elements of the field so that $\FF_q = \{x_1, \dots, x_q\}$.
 Consider the map
 \[
  Ev: \fq[X]_{\leq k} \to \FF_q^{q+1}: f \mapsto (f(x_1), \dots, f(x_q,1), f(\infty)).
 \]
 Then the image of $Ev$ is a linear $[q+1,k+1]_q$ code, called the \emph{extended Reed-Solomon code}, and denoted $ERS(q,k)$.
\end{df}

We refer the reader to \cite[\S 6.2]{Ball} for more on (extended) Reed-Solomon codes.
The most important property is that they are MDS codes.

\begin{res}[see e.g.\ {\cite[Lemma 6.5]{Ball}}]
 For each $k < q$, the extended Reed-Solomon code $ERS(q,k)$ has minimum weight $q+2-k$, and hence is an MDS code.
\end{res}

We call a matrix $G \in \FF_q^{k \times n}$ a \emph{generator matrix} for the linear $[n,k]_q$ code $C$ if the rows of $G$ form a basis for $C$.
Let $P_i$ denote the 1-dimensional subspace of $\FF_q^k$ spanned by the $i$\th column of $G$.
Then $(P_1, \dots, P_n)$ can be viewed as a collection of points in the projective space $\pg(k-1,q)$.
We call $(P_1, \dots, P_n)$ a \emph{projective system} associated with $C$.
Note that all points $P_i$ are distinct because $C$ is assumed to be projective.
The set of generator matrices of $C$ equals $\sett{M G}{M \in \GL(k,q)}$.
It readily follows that the set of projective systems associated with $C$ forms an orbit of the action of $\PGL(k,q)$ on ordered $n$-tuples of points of $\pg(k-1,q)$.
We typically do not care about the order of the points of the projective system:
Two linear codes $C$ and $C'$ are said to be \emph{equivalent} if $C$ can be transformed into $C'$ by coordinate transformations and scaling in each coordinate position, i.e.\ 
\[
 C' = \sett{\big(\alpha_1 c(\sigma(1)), \ldots, \alpha_n c(\sigma(n))\big)}{c \in C}
\]
for some permutation $\sigma$ on $\set{1,\dots,n}$, and some non-zero scalars $\alpha_1, \dots, \alpha_n \in \fq^*$.
It is easy to check that two codes are equivalent if and only if they correspond to the same unordered projective system.

\bigskip

The projective systems of the extended Reed-Solomon codes will play a central role in this paper.

\begin{df}
 \label{Df:nu}
 Let $\mp_1$ and $\mp_k$ denote the sets of points of $\pg(1,q)$ and $\pg(k,q)$ respectively.
 Define the map
 \[
  \nu_k: \mp_1 \to \mp_k: (x:y) \mapsto (x^k : x^{k-1} y : \ldots : x y^{k-1} : y^k ).
 \]
 The image $\mn$ of $\nu_k$ is called the \emph{canonical normal rational curve} in $\pg(k,q)$.
 Any image of $\mn$ under the group of collineations $\PGL(k+1,q)$ is called a \emph{normal rational curve}.
\end{df}

More information on normal rational curves can be found in \cite[\S 21.1]{Hirschfeld:85} and \cite[\S 6.5]{Hirschfeld:Thas}.
Using the fact that $Ev(X^k), \dots, Ev(X), Ev(1)$ forms a basis for $ERS(q,k)$, the link between extended Reed-Solomon codes and normal rational curves follows easily.

\begin{lm}
 The projective systems associated to the extended Reed-Solomon code $ERS(q,k)$ are the normal rational curves in $\pg(k,q)$.
\end{lm}

The fact that $ERS(q,k)$ is an MDS code is equivalent to normal rational curves being so-called \emph{arcs} in $\pg(k,q)$.
This means the following.

\begin{res}[{\cite[Theorem 21.1.1 (iv)]{Hirschfeld:85}}]
 \label{Res:Arc}
 Let $\mn$ be a normal rational curve in $\pg(k,q)$.
 Any subspace of dimension $r < k$ of $\pg(k,q)$ contains at most $r+1$ points of $\mn$.
\end{res}

Lastly, we remark an important connection between a linear code $C$ and the geometry of its projective system $\ms$ arising from a generator matrix $G$ of $C$.
For a vector $v \in \FF_q^k$, let $v^\perp$ denote the subspace $\sett{w \in \FF_q^k}{v^\top w = 0}$.
We can interpret $v^\perp$ as a hyperplane of $\pg(k-1,q)$ (or the whole space if $v = \zero$).
A codeword $c \in C$ can uniquely be written as $v^\top G$ for some $v \in \FF_q^k$.
Note that $\wt(c)$ equals the number of points $P_i$ not contained in $v^\perp$.
Hence, the distribution of the points $P_1, \dots, P_n$ over the hyperplanes of $\pg(k-1,q)$ tells us the weight distribution of the code $C$.

\section{The EKR module property for linear codes}
 \label{Sec:Mod prop}

In this section we start the investigation of intersection problems for linear codes $C$.
The relation between $C$ and its projective system $\ms$ will be crucial.
As an example, the reader can keep the case where $C = ERS(q,k)$ is an extended Reed-Solomon code, and $\ms$ is a normal rational curve in mind.

We start with a simple observation.

\begin{prop}
 \label{Prop:Weak EKR:t}
 Suppose that $C$ is a linear $[n,k]_q$ code which has a linear $[n,t]_q$ MDS subcode.
 Then the largest $t$-intersecting families of $C$ have size $q^{k-t}$.
\end{prop}

\begin{proof}
 Suppose that $C' \leq C$ is a linear $[n,t,n-t+1]_q$ MDS subcode.
 Then no two elements of any additive coset of $C'$ are $t$-intersecting.
 Therefore, a $t$-intersecting family $\mf$ in $C$ contains at most 1 element from each of the $q^{k-t}$ additive cosets of $C'$ in $C$.
 This proves that $|\mf| \leq q^{k-t}$.

 Now pick $t$ distinct coordinate positions $i_1, \dots, i_t$, and $t$ (not necessarily distinct) scalars $\alpha_1, \dots, \alpha_t \in \fq$.
 Define the $t$-star 
 \[
  \mf = \sett{c \in C}{(\forall j \in \{1, \dots, t\})(c(i_j)=\alpha_j)}.
 \]
 Then $\mf$ is a subspace of $C$ of co-dimension at most $t$, hence contains at least $q^{k-t}$ elements.
 Since it is $t$-intersecting, we must have $|\mf| = q^{k-t}$.
\end{proof}

This establishes the weak EKR property.

\begin{crl}
 \label{Prop:Weak EKR}
 If $C$ is a linear $[n,k]_q$ code, then either $C$ is an intersecting family, or the largest intersecting families have size $q^{k-1}$.
\end{crl}

\begin{proof}
 Either $C$ has no codewords of weight $n$, in which case it is an intersecting family, or it does have a codeword $c$ of weight $n$, in which case $\vspan c$ is a $[n,1]_q$ MDS subcode of $C$, and \Cref{Prop:Weak EKR:t} finishes the proof.
\end{proof}

Next, we want to investigate whether the EKR module property holds.
In order to do so, we first determine the eigenvalues of relevant graphs.

\begin{df}
 \label{Df:Gamma_T}
 \begin{enumerate}
  \item \label{Df:Gamma_T:1} Let $C$ be a linear $[n,k]_q$ code, and $T \subseteq \{0,\dots,n-1\}$. Define the graph \[
   \Gamma_T(C) = \Cay(C, \sett{c \in C}{n - \wt(c) \in T}),
  \]
  on $C$, where $c$ and $c'$ are adjacent if and only they coincide in exactly $t$ coordinates for some $t \in T$.
  \item \label{Df:Gamma_T:2} Suppose that $G$ is a generator matrix of $C$, and $H$ is a right inverse.
  For each $v \in \FF_q^k \setminus \set \zero$ and $\alpha \in \FF_q$, define the vector
  \[
   \xi_{v,\alpha}: C \to \RR: c \mapsto \begin{cases}
    1 & \text{if } c H v = \alpha, \\
    0 & \text{otherwise.}
   \end{cases}
  \]
  Note that $\xi_{v,\alpha} = \xi_{\beta v, \beta \alpha}$ for each non-zero scalar $\beta$.
 \end{enumerate}
\end{df}

Note that if $c$ is a codeword of $C$, then $c H v$ is a linear condition on $c$ in the form $\sum_{i=1}^n (Hv)(i) c(i) = \alpha$.
In particular, if $Hv$ is a standard basis vector, then the condition becomes $c(i) = \alpha$ for a certain $i$.
This means that $\xi_{v,\alpha}$ is the characteristic vector of a star.
Since $H$ is a right inverse of $G$, this happens if and only if $v$ is a column of $G$, in which case $\vspan v$ is in the projective system of $C$ arising from $G$.
Note also that since $H$ is a right inverse of $G$, $\chi_{v,\alpha}$ and $\chi_{w,\beta}$ will be linearly independent if $v$ and $w$ are linearly independent.

We will be interested in the component of $\xi_{v,\alpha}$ orthogonal to $\one$, which equals $\xi_{v,\alpha} - \frac 1q \one$.
Note that $\sett{\xi_{v, \alpha}}{\alpha \in \FF_q}$ is linearly independent, since the vectors have disjoint support.
However, their projections onto $\one^\perp$ are no longer linearly independent since
\begin{align}
 \label{Eq:xi_v alpha}
 \sum_{\alpha \in \FF_q} \left( \xi_{v,\alpha} - \frac 1q \one \right)
 = \left( \sum_{\alpha \in \FF_q} \xi_{v,\alpha} \right) - \left( \sum_{\alpha \in \FF_q} \frac 1q \one \right) = \one - \one = \zero.
\end{align}
However, since we project $q$ linearly independent vectors onto a hyperplane, their projections must still span a $(q-1)$-dimensional space, that is
\[
 \dim \vspann{\xi_{v, \alpha} - \frac 1q \one}{\alpha \in \FF_q} = q-1.
\]
Therefore, (\ref{Eq:xi_v alpha}) describes the only linear dependence between the vectors, and any subset $q-1$ of the vectors $\xi_{v,\alpha} - \frac 1q \one$ must be linearly independent.

\begin{prop}
 \label{Prop:Eigval Gamma_T}
 Let $C$ be a linear $[n,k]_q$ code with generator matrix $G$, and projective system $\ms$.
 Choose $T \subseteq \{0, \dots, n-1\}$.
 Then the set $\set \one \cup \sett{\xi_{v,\alpha} - \frac 1q \one}{v \in \FF_q^k \setminus \zero, \, \alpha \in \FF_q^*}$ is a basis of eigenvectors for $A_{\Gamma_T(C)}$.
 Let $\mm$ be the set of hyperplanes $\Pi$ with $|\Pi \cap \ms| \in T$.
 Then the eigenvalue of $\one$ equals $(q-1)|\mm|$, and the eigenvalue of $\xi_{v,\alpha} - \frac 1q \one$ equals $q |\sett{\Pi \in \mm}{\vspan v \in \Pi}| - |\mm|$.
\end{prop}

\begin{proof}
 The group $(C,+)$ is isomorphic to $(\FF_q^k,+)$.
 An explicit isomorphism is given by $f: \FF_q^k \to C: v \mapsto v^\top G$.
 Note that if $H$ is a right inverse for $G$, then $f^{-1}(c) = (cH)^\top$.
 Therefore, $\Gamma_T(C)$ is isomorphic to the Cayley graph $\Gamma' = \Cay(\FF_q^k,S_T)$ with
 \[
  S_T =  \sett{v \in \FF_q^k}{n - \wt(v^\top G) \in T}.
 \]
 It is easy to check that $S_T$ is closed under taking non-zero scalar multiples.
 Hence, we can use (\ref{Eq:Characters 2}) to determine the eigenvalues of $\Gamma'$.
 Construct the set $S'_T$ by choosing a unique scalar multiple of every vector in $S_T$.
 Note that there is a one-to-one correspondence between the elements of $S'_T$ and the elements of $\mm$.
 Indeed, $v \in S_T$ if and only if $v^\perp \in \mm$.
 Since $S'_T$ contains a unique scalar multiple of every vector of $S_T$, for every hyperplane $\Pi \in \mm$, there is a unique $v \in S'_T$ with $\Pi = v^\perp$.

 Take a non-trivial character $\chi_v$ of $\FF_q^k$, as defined in \Cref{Sec:Eigval Cayley}.
 Then by (\ref{Eq:Characters 2}),
 \[
  \chi_v(S_T) = q |S'_T \cap v^\perp| - |S_T'|
   = q |\sett{\Pi \in \mm}{v \in \Pi} - |\mm|.
 \]
 It follows that $\chi_v \circ f^{-1}$ must form a basis of eigenvectors of $A_{\Gamma_T(C)}$.
 To finish the proof, it suffices to show that for every non-zero vector $v \in \FF_q^k$,
 \[
  \vspann{\chi_{\alpha v} \circ f^{-1}}{\alpha \in \FF_q^*}
  = \vspann{\xi_{v,\alpha} - \frac 1q \one}{\alpha \in \FF_q^*},
 \]
 or equivalently that
 \[
  \vspann{\chi_{\alpha v}}{\alpha \in \FF_q^*}
  = \vspann{\left( \xi_{v,\alpha} - \frac 1q \one \right) \circ f}{\alpha \in \FF_q^*}
 \]
 Since all vectors listed above are orthogonal to $\one$, it suffices to show that
 \[
  \vspan{\one, \vspann{\chi_{\alpha v}}{\alpha \in \FF_q^*}}
  = \vspan{\one,\vspann{\left( \xi_{v,\alpha} - \frac 1q \one \right) \circ f}{\alpha \in \FF_q^*}}.
 \]
 This means that
 \[
  \vspann{\chi_{\alpha v}}{\alpha \in \FF_q} = \vspann{\xi_{v,\alpha} \circ f}{\alpha \in \FF_q}.
 \]
 Since both spaces have dimension $q$, it suffices to check that the left one is a subspace of the right one.
 Note that $\xi_{v,\alpha} \circ f(w)$ equals 1 if $w^\top G H v = w^\top v = \alpha$, and 0 otherwise.
 It follows that $\vspann{\xi_{v,\alpha} \circ f}{\alpha \in \FF_q}$ is the space of functions $g:\FF_q^k \to \CC$ for which $g(w)$ only depends on $v^\top w$.
 The fact that $\chi_{\alpha v}(w) = \zeta^{T(\alpha v \cdot w)}$ only depends on $v^\top w$ finishes the proof.
\end{proof}

Now we are ready to describe how the EKR module property holds, possibly after a slight alteration.

\begin{df}
 Suppose that $C$ is a linear $[n,k]_q$ code with generator matrix $G$.
 Choose an integer $n^+ \geq n$.
 Let $G^+$ be a $k \times n^+$ matrix over $\FF_q$ that has $G$ as a submatrix.
 Let $C^+$ be the rowspace of $G^+$.
 Then we call $C^+$ an \emph{extension} of $C$.
 For each codeword $c = v^\top G$ of $C$, we call $c^+ = v^\top G^+$ the \emph{extension} of $c$.
 Note that every projective system $\ms^+$ of $C^+$ has a subset $\ms$ which is a projective system of $C$.
\end{df}

\begin{prop}
 \label{Prop:Module prop}
 Let $C$ be a linear $[n,k]_q$ code with projective system $\ms$.
 Assume that $C$ is not an intersecting family.
 \begin{enumerate}
  \item \label{Prop:Module prop:1} $C$ has the EKR module property if and only if every point not on $\ms$ in $\pg(k-1,q)$ is incident with a hyperplane that doesn't intersect $\ms$.
  \item \label{Prop:Module prop:2} There exists an extension $C^+$ of $C$, unique up to code equivalence, such that two codewords of $C$ intersect if and only if their extensions in $C^+$ intersect, and $C^+$ has the EKR module property.
 \end{enumerate}
\end{prop}

\begin{proof}
Choose a generator matrix $G$ for $C$, and let $H$ be a right inverse of $G$.
Denote the $i$\th column of $G$ by $g_i$.

Consider the graph $\Gamma = \Gamma_{\{0\}}(C)$ on $C$, where two codewords are adjacent if and only if they are not intersecting.
Then intersecting families in $C$ are equivalent to cocliques in $\Gamma$.
Let $\mm$ be the set of hyperplanes that don't intersect $\ms$.
The assumption that $C$ is not an intersecting family is equivalent to $\mm$ not being empty.
We can determine the eigenvalues of $\Gamma$ using \Cref{Prop:Eigval Gamma_T}.
The eigenvalue of $\one$ equals $(q-1)|\mm|$, and the other eigenvalues equal $q|\sett{\Pi \in \mm}{P \in \Pi}| - |\mm|$, with $P$ the points of $\pg(k-1,q)$.
This value is minimal if $P$ is not on any hyperplanes of $\mm$, which happens for the points of $\ms$.
Hence, the Hoffman ratio bound tells us that if $\mf$ is an intersecting family, then
\[
 |\mf| \leq \frac{q^k}{\frac{(q-1)|\mm|}{|\mm|}+1} = q^{k-1},
\]
with equality if and only if the characteristic vector $\chi_\mf$ of $\mf$ is in the span of $\one$ and the $-|\mm|$-eigenspace of $A_\Gamma$.

Now consider a star $\mf_{i,\alpha} = \sett{c \in C}{c(i) = \alpha}$.
It is easy to check that $\chi_{\mf_{i,\alpha}} = \xi_{g_i,\alpha}$, with $\xi_{g_i, \alpha}$ as in \Cref{Df:Gamma_T} (\ref{Df:Gamma_T:1}).
It follows that the EKR module property holds if and only if the $-|\mm|$-eigenspace is spanned by the vectors $\xi_{g_i,\alpha}-\frac 1q \one$, with $i\in\{1,\dots,n\}$ and $\alpha \in \FF_q$.
By \Cref{Prop:Eigval Gamma_T}, this is equivalent to the points of $\ms$ being the only ones that are not incident to any hyperplane of $\mm$.
This proves (\ref{Prop:Module prop:1}).

Now suppose that $\ms^+$ is the set of all points that are not on any hyperplane of $\mm$.
Then clearly $\ms \subseteq \ms^+$.
Let $C^+$ be the extension of $C$ corresponding to projective system $\ms^+$ (note that $C^+$ is only defined up to code equivalence).
Then $c = v^\top G$ and $c' = w^\top G$ in $C$ are intersecting if and only if $c-c'$ has no zeros, or equivalently, the hyperplane $(v-w)^\perp$ doesn't contain any point of $\ms$.
Then by the definition of $\ms^+$, $(v-w)^\perp$ doesn't contain any point of $\ms^+$.
Hence, codewords in $C$ are intersecting if and only if their extensions in $C^+$ are intersecting.
Furthermore, by construction of $\ms^+$, every point of $\pg(k-1,q)$ not in $\ms^+$ is incident to some hyperplane of $\mm$.
This implies that $C^+$ has the EKR module property.

Note also that $\ms^+$ is the only projective system extending $\ms$ yielding the desired code $C^+$.
Any other projective system $\ms'$ either contains some point outside of $\ms^+$, which then implies that some codewords $c$ and $c'$ are not intersecting, but their extensions are, or $\ms'$ is a proper subset of $\ms^+$, in which case the EKR module property doesn't hold.
This proves (\ref{Prop:Module prop:2}).
\end{proof}

\begin{crl}
 \label{Crl:Small proj sys}
 Let $C$ be a linear $[n,k]_q$ code with $k \geq 4$ and $n \leq q+1$.
 Then the EKR module property holds.
\end{crl}

\begin{proof}
 Let $\ms$ be the projective system associated to $C$ in $\pg(k-1,q)$.
 Take a point $P \notin \ms$.
 Let $\ms'$ be the projection of $\ms$ from $P$ as a set of points in $\pg(k-2,q)$.
 Then $|\ms'| \leq q+1$.
 If $\ms'$ would intersect every hyperplane of $\pg(k-2,q)$, then $\ms'$ would need to consist of the points on a line by \cite{Bose:Burton}.
 However, then $\ms$ would be contained in a plane through $P$, and cannot span the whole space $\pg(k-1,q)$, a contradiction.
 We conclude that some hyperplane through $P$ does not intersect $\ms$.
\end{proof}

In particular, this implies that the EKR module property holds for the extended Reed-Solomon codes $ERS(q,k)$ with $q > k \geq 3$.
However, it does not hold for $k=3$ if $q$ is even, as shown in the following example.

\begin{ex}
 \label{Ex:Nucleus}
 We consider the extended Reed-Solomon code $C = ERS(q,2)$ with $q > 2$ an even prime power.
 The associated projective system $\ms$ is a normal rational curve, see \Cref{Df:nu}, which in our case is given by
 \[
  \ms = \sett{(1:x:x^2)}{x \in \fq} \cup \{(0:0:1)\}.
 \]
 Then $\ms$ is the quadric defined by equation $Y^2 = XZ$.
 Since $q$ is even, $\ms$ has a \emph{nucleus}, namely the point $N = (0:1:0)$, see e.g.\ \cite[Chapter 8]{Hirschfeld:79}.
 Every line through $N$ intersects $\ms$ in a unique point.
 Therefore, $N$ is not incident with any of the hyperplanes (which in this case are lines) that don't intersect $\ms$.
 The unique extension $C^+$ of $C$ described in \Cref{Prop:Module prop} (\ref{Prop:Module prop:2}) arises from extending $\ms$ to $\ms^+ = \ms \cup \set{N}$.

 Algebraically, this corresponds to the fact that if $f = aX^2 + bXY + cY^2$ and $g = a'X^2 + bXY + c'Y^2$ are polynomials in $\hom 2$ with the same $XY$-coefficient $b$, then $f-g = (a-a') X^2 + (c-c') Y^2= \left( \sqrt{a-a'} X + \sqrt{c-c'} Y \right)^2$.
 Note that square roots are well-defined in fields of even order.
 Since $\sqrt{a-a'} X + \sqrt{c-c'} Y$ has a non-zero root, $f$ and $g$ are intersecting.
 This explains why in \Cref{Res:Hom Pol}, we need to also consider the intersecting families 
 \[ \sett{aX^2 + bXY + cY^2}{a,c \in \fq}.
 \]
\end{ex}

Given a linear code $C$, we established geometric properties of its projective system $\ms$, which are equivalent to the weak EKR and EKR module property for $C$.

\begin{prob}
Is there a geometric property of $\ms$ that is equivalent to the strict EKR property for $C$?
\end{prob}

The answer is not obvious.
In the next section, we give a sufficient condition for the strict EKR property to hold, and in fact to even prove stability results.

\section{The strict EKR property and beyond for linear codes}
 \label{Sec:Strict prop}

As discussed in \Cref{Sec:Overview}, this section consists of four parts.
We consider a linear code $C$, and a large intersecting family $\mf$ in $C$.
Let $\mf_{i,\alpha}$ denote the star $\sett{c \in C}{c(i) = \alpha}$.
In \Cref{Sec:Few or many}, we prove that every star shares either few or many codewords with $\mf$, excluding a middle spectrum for the possible intersection sizes.
In \Cref{Sec:More than few}, we prove that some star $\mf_{i,\alpha}$ intersects $\mf$ in ``more than few'' codewords, which then implies that it shares many codewords with $\mf$.
If $\mf$ is large enough, this suffices to conclude that $\mf \subseteq \mf_{i,\alpha}$.
In \Cref{Sec:Strict sub}, we use these results to prove \Cref{Thm:Codes} (\ref{Thm:Codes:3}), which is a sufficient condition for the strict EKR property to hold in $C$.
In \Cref{Sec:Stab}, we prove \Cref{Thm:Stability}, which yields a stability theorem for large intersecting families in $C$.

\subsection{Intersecting every star in few or many codewords}
 \label{Sec:Few or many}

 In this section, we will prove the following theorem.

\begin{thm}
 \label{Thm:Few or many}
 Consider a linear $[n,k]_q$ code $C$ with projective system $\ms$.
 Suppose that $\mf$ is an intersecting family in $C$.
 Choose $i \in \{1,\dots,n\}$ and $\alpha \in \FF_q$, and let $\mf_{i,\alpha}$ be the star $\sett{c \in C}{c(i) = \alpha}$.
 Let $Q \in \ms$ be the point corresponding to coordinate position $i$.
 Define $\mm$ to be the set of hyperplanes of $\pg(k-1,q)$ that don't intersect $\ms$.
 For each line $\ell$ through $Q$, define
 \[
  \lambda_\ell = q \sum_{P \in \ell \setminus \set Q} \left( |\sett{\Pi \in \mm}{P \in \Pi}| - \frac{|\mm|}{q} \right)^2,
 \]
 and let $\lambda$ be the maximum value of $\lambda_\ell$ over all lines $\ell$ through $Q$.
 Then
 \[
  |\mf \cap \mf_{i,\alpha}| \cdot |\mf \setminus \mf_{i,\alpha}| \leq \lambda \left(\frac{q^{k-1}}{|\mm|}\right)^2
 \]
\end{thm}
We prove \Cref{Thm:Few or many} by an application of the bipartite expander mixing lemma to the following graph.

\begin{df}
 \label{Df:Bip graph}
 Given a linear $[n,k]_q$ code $C$, an integer $i \in \{1,\dots,n\}$, and a scalar $\alpha \in \FF_q$, define the bipartite graph $B(C,i,\alpha)$ as having vertex sets $L = \sett{c \in C}{c(i) = \alpha}$ and $R = C \setminus L$, where $c \in L$ is adjacent to $c' \in R$ if and only if $c$ and $c'$ don't intersect, i.e.\ $\wt(c-c') = n$.
\end{df}

As a first step, let us check that $B(C,i,\alpha)$ is biregular, which is a necessary condition for applying the bipartite expander mixing lemma. 

\begin{lm}
 \label{Lm:Biregular}
 Use the notation of \Cref{Df:Bip graph}.
 \begin{enumerate}
  \item \label{Lm:Biregular:Isomorphic} For $\alpha, \alpha' \in \FF_q$, the graphs $B(C,i,\alpha)$ and $B(C,i,\alpha')$ are isomorphic.
  \item \label{Lm:Biregular:Sharply} The automorphism group of $B(C,i,\alpha)$ has subgroups $G_L$ and $G_R$ that act sharply transitively on $L$ and $R$ respectively.
 Therefore, $B(C,i,\alpha)$ is biregular.
 \end{enumerate}
\end{lm}

\begin{proof}
 (1) It is easy to check that if $c \in C$ with $c(i) = \alpha'-\alpha$, then the map
 \[
  f: C \to C: c' \mapsto c'+c
 \]
 is an isomorphism from $B(C,i,\alpha)$ to $B(C,i,\alpha')$.

 (2) Using (1), it suffices to prove the statement for $\alpha = 0$.
 Define $G_L = \sett{c \in C}{c(i) = 0}$ and
 define the group $(G_R,*)  = (\FF_q^*,\cdot) \ltimes (G_L,+)$, with operation
 \[
  (\beta,c') * (\gamma,c'') = (\beta \gamma, c' + \beta c'').
 \]
 Then $G_R$ acts on $C$ as follows:
 \[
  (\beta,c): C \to C: c' \mapsto \beta c' + c.
 \]
 It is easy to see that this action stabilises $L$ and $R$, and preserves adjacency.
 If $c, c' \in R$, then $\left(\frac{c'(i)}{c(i)}, c' - \frac{c'(i)}{c(i)} c\right)$ is the unique element of $G_R$ sending $c$ to $c'$, so $G_R$ acts sharply transitively on $R$.
 Moreover, the subgroup $G_L$ of $G_R$ acts sharply transitively on $L$.
\end{proof}

We can use $B(C,i,\alpha)$ to define a distance 2 multigraph $B^2$, where the multiplicity of edge $\{c,c'\}$ equals the number of walks of length two from $c$ to $c'$.
The previous lemma implies that $B^2$ is the disjoint union of two Cayley multigraphs, one of which is abelian.
This allows us to determine the eigenvalues of the adjacency matrix of $B^2$ using characters, as described in \Cref{Sec:Eigval Cayley}.
This in turn reveals the eigenvalues of the adjacency matrix of $B(C,i,\alpha)$.

\begin{prop}
 \label{Prop:Eigval B}
 Suppose that $C$ is a linear $[n,k]_q$ code with projective system $\ms$.
 Suppose that $Q \in \ms$ corresponds to coordinate $i$ of $C$.
 Let $\mm$ denote the set of hyperplanes that don't intersect $\ms$.
 Then the spectrum of the adjacency matrix of $B(C,i,0)$ can be constructed as follows:
 \begin{itemize}
  \item Add the eigenvalues $\pm \sqrt{q-1} |\mm|$ with multiplicity 1,
  \item for each line $\ell$ through $Q$, add the eigenvalues
  \[
   \pm \left(q \sum_{P \in \ell \setminus \set Q} \left(|\sett{\Pi \in \mm}{P \in \mm}| - \frac{|\mm|}q \right)^2 \right)^{1/2}
  \]
  with multiplicity $q-1$,
  \item add eigenvalue $0$ with multiplicity $(q-2) q^{k-1}$.
 \end{itemize}
\end{prop}

\begin{proof}
Let $B$ denote the graph $B(C,i,0)$.
Following \Cref{Df:Bip graph}, define $L = \sett{c \in C}{c(i) = 0}$ and $R = C \setminus L$.
Let $M$ be the submatrix of $A_B$ with rows labelled by $L$ and columns by $R$.
Then
\begin{align*}
 A_B = \begin{pmatrix}
  O & M \\ M^\top & O
 \end{pmatrix}, &&
 A_B^2 = \begin{pmatrix}
  M M^\top & O \\ O& MM^\top
 \end{pmatrix}.
\end{align*}
The eigenvalues of $A_B$ are (plus or minus) the square roots of the eigenvalues of $A_B^2$.
It is well known that since $B$ is bipartite, for any eigenvalue $\lambda$ of $A_B$, $-\lambda$ is an eigenvalue with the same multiplicity.
Moreover, $M M^\top$ and $M^\top M$ have the same non-zero eigenvalues with the same multiplicity.
These facts together imply that for any non-zero number $\lambda$, $\lambda$ is an eigenvalue of $A_B$ with multiplicity $m$ if and only if $\lambda^2$ is an eigenvalue of $M M^\top$ with multiplicity $m$.
Thus, it suffices to determine the spectrum of $M M^\top$.

For any sequence of vertices $c_1, \dots, c_k$ of $B$, let $B(c_1,\dots,c_k)$ denote the set of common neighbours of $c_1, \dots, c_k$ in $B$.
Then $M M^\top(c_1,c_2) = |B(c_1,c_2)|$.
Note that if $c_3 \in R$, then $c_3 \in B(c_1,c_2)$ if and only if $\wt(c_3-c_1) = \wt(c_3-c_2) = \wt((c_3-c_1) - (c_2 - c_1)) = n$, which is equivalent to $c_3 - c_1 \in B(\zero, c_2-c_1)$.
In other words, $M(c_1,c_2) = M(\zero,c_2-c_1)$.
Therefore, we can see $M M^\top$ as the adjacency matrix of a Cayley multigraph $\Gamma$ on $(L,+)$, where the underlying multiset has multiplicity function $\mu(c) = M M^\top(\zero,c) = |B(\zero,c)|$.
Moreover, if $c_2 \in R$, then $c_2 \in B(\zero,c_1)$ if and only if $\alpha c_2 \in B(\zero,\alpha c_1)$ for any non-zero scalar $\alpha$.
Hence, $\mu(\alpha c) = \mu(c)$ for any non-zero scalar $\alpha$ and codeword $c \in L$.
This means that we can use (\ref{Eq:Characters 2}) to determine the eigenvalues of $M M^\top$.

The graph $B$ is $(d_L, d_R)$-biregular with $d_L = (q-1) |\mm|$ and $d_R = |\mm|$.
Indeed, the fact that $B$ is biregular follows from \Cref{Lm:Biregular} (\ref{Lm:Biregular:Sharply}).
The degree of $\zero$ in $B$ equals the number of codewords of $C$ of weight $n$.
Every codeword of weight $n$ corresponds to a hyperplane of $\pg(k-1,q)$ that doesn't intersect $\ms$, and vice versa, every such hyperplanes corresponds to $q-1$ codewords of weight $n$.
This proves that $d_L = (q-1)|\mm|$.
Double counting the edges of $B$ tells us that $|L| d_L = |R| d_R$.
Since $|L| = q^{k-1}$ and $|R| = q^k - q^{k-1}$, $d_R = \frac{|L|}{|R|} d_L = \frac{1}{q-1} (q-1)|\mm| = |\mm|$.
It follows that the trivial character of $(L,+)$, which yields the eigenvector $\one$, has eigenvalue $\sum_{c \in L} |B(\zero,c)|$.
This is easily seen to equal the number of walks of length 2 starting in $\zero$.
Since every neighbour of $\zero$ has degree $d_R$, this equals $d_L d_R = (q-1) |\mm|^2$.

Now we investigate the non-trivial characters of $(L,+)$.
We make some assumptions on $C$.
We may suppose that $i=1$ (if necessary, permute the coordinates of $C$, which yields an equivalent code).
Since we assumed that $C$ is projective, no generator matrix of $C$ has a zero column.
Now take a basis $c_2, \dots, c_k$ of $L$, choose $c_1 \in R$, and consider the generator matrix $G$ with rows $c_1, \dots, c_k$.
We may suppose that this is the generator matrix used to construct $\ms$ (switching between generator matrices of $C$ corresponds to applying projectivities to the corresponding projective system).
Note that this means that $Q = (1:0:\ldots:0)$.

Consider the map $f: \FF_q^k \to C: v \mapsto v^\top G$.
This is an isomorphism of vector spaces.
Define $V = \sett{v \in \FF_q^k}{v(1) = 0} = Q^\perp$.
Then the restriction of $f$ to $V$ is a vector space isomorphism (and hence a group isomorphism) from $V$ to $L$.
Make a set $V' \subset V \setminus \set \zero$ that contains a unique scalar multiple of each vector in $V$.
It follows from \Cref{Sec:Eigval Cayley} and (\ref{Eq:Characters 2}) that the characters of $L$ are given by $\chi_v$ with $v \in V$ defined by $\chi_v(w^\top G) = \zeta^{T(v^\top w)}$, with $\zeta$ and $T$ as in \Cref{Sec:Eigval Cayley}.
The eigenvalue of $\chi_v$ equals
\begin{equation}
 \label{Eq:Eigval B:1}
 \chi_v(\Gamma) = |B(\zero)| + q \sum_{w \in V' \cap v^\perp} |B(\zero, w^\top G)| - \sum_{w \in V'} |B(\zero,w^\top G)|.
\end{equation}
We know that $|B(\zero)| = d_L = (q-1)d_R$.
Moreover, since $|B(\zero,\alpha c)| = |B(\zero,c)|$ for all non-zero scalars $\alpha$,
\begin{align*}
 \sum_{w \in V'} |B(\zero,w^\top G)|
 = \sum_{w \in V'} \frac{1}{q-1}  \sum_{\alpha \in \FF_q^*} |B(\zero,\alpha w)|
 = \frac 1{q-1} \sum_{c \in L \setminus \set \zero} |B(\zero,c)|.
\end{align*}
The latter sum equals the number of walks of length 2 that start in $\zero$ but do not end in $\zero$.
Hence, the sum equals $d_L(d_R-1) = (q-1) d_R (d_R-1)$.
Therefore,
\begin{equation}
 \label{Eq:Eigval B:2}
 |B(\zero)| - \sum_{w \in V'} |B(\zero,w^\top G)| = (q-1) d_R - \frac{1}{q-1} (q-1) d_R (d_R-1) = - d_R (d_R-q).
\end{equation}

Next, we examine the sum $\sum_{w \in V' \cap v^\perp} |B(\zero, w^\top G)|$.
We can rewrite it as
\begin{align}
 \label{Eq:Eigval B:3}
 \begin{split}
 \sum_{w \in V' \cap v^\perp} &|B(\zero, w^\top G)|
 = |\sett{(w,u) \in (V' \cap v^\perp) \times \FF_q^k}{u^\top G \in B(\zero,w^\top G)}| \\
 &= |\sett{(w,u) \in (V' \cap v^\perp) \times \FF_q^k}{\wt(u^\top G) = \wt((u-w)^\top G) = n}| \\
 & = |\sett{(w,u_1,u_2) \in (V' \cap v^\perp) \times \FF_q^k \times \FF_q^k}{\wt(u_1^\top G) = \wt(u_2^\top G) = n, \, u_2 = u_1 - w}| \\
 &= |\underbrace{\sett{(u_1,u_2) \in \FF_q^k \times \FF_q^k}{\wt(u_1^\top G) = \wt(u_2^\top G) = n, \, u_1 - u_2 \in V' \cap v^\perp}}_{=E_v}|
 \end{split}
\end{align}
Note that $\wt(u_1^\top G) = n$ if and only if $u_1^\perp$ is a hyperplane in $\mm$.
Suppose now that we take two hyperplanes $\Pi_1, \Pi_2 \in \mm$.
How many pairs $(u_1,u_2) \in E_v$ are there with $\Pi_1 = u_1^\perp$ and $\Pi_2 = u_2^\perp$?
First note that $\Pi_1$ and $\Pi_2$ should be distinct.
Otherwise $u_1 - u_2$ is a multiple of $u_1$, but no non-zero multiple of $u_1$ is in $V$, so no multiple of $u_1$ is in $V'$.
Now suppose that $\Pi_1$ and $\Pi_2$ are distinct, hence $u_1$ and $u_2$ are linearly independent.
Then $\vspan{u_1,u_2} \cap V = \vspan{u_1 - \frac{u_1(1)}{u_2(1)} u_2}$.
There is a unique scalar multiple $u_3$ of the vector $u_1 - \frac{u_1(1)}{u_2(1)} u_2$ in $V'$.
Then $u_3$ is in $v^\perp$ if and only if the line $\vspan{u_1,u_2}$ intersects the hyperplanes $V$ and $v^\perp$ in the same projective point (which must be $\vspan{u_3}$), or $\vspan{u_1,u_2}$ is contained in $v^\perp$.
We claim that this is equivalent to $\Pi_1 \cap \ell = \Pi_2 \cap \ell$, with $\ell = \vspan{Q,v}$:
First consider the case where $\vspan{u_1,u_2}$ is not contained in $v^\perp$, and $\vspan{u_1,u_2} \cap V = \vspan{u_1,u_2} \cap v^\perp$.
This is equivalent to $\vspan{u_1^\perp \cap u_2^\perp, V^\perp} = \vspan{u_1^\perp \cap u_2^\perp, v}$.
Note that this means that $\vspan{\Pi_1 \cap \Pi_2, Q} = \vspan{\Pi_1 \cap \Pi_2, v}$.
Since $\Pi_1 \cap \Pi_2$ has codimension 2, and $Q$ and $\vspan v$ are points outside $\Pi_1 \cap \Pi_2$, this happens if and only if the line $\ell = \vspan{Q,v}$ intersects $\Pi_1 \cap \Pi_2$, or equivalently, $\Pi_1 \cap \ell = \Pi_2 \cap \ell$.
The other case to consider is when $u_1, u_2 \in v^\perp$.
Then $v \in \Pi_1$ and $v \in \Pi_2$.
Since $Q \notin \Pi_1$ and $Q \notin \Pi_2$, this means that $\Pi_1$ and $\Pi_2$ intersect $\ell = \vspan{Q,v}$ in the same point $\vspan v$.
We conclude that the number of pairs $(u_1,u_2) \in E_v$ with $\Pi_1 = u_1^\perp$ and $\Pi_2 = u_2^\perp$ equals 1 if $\Pi_1$ and $\Pi_2$ are distinct hyperplanes intersecting $\ell$ in the same point, and zero otherwise.
We plug this into (\ref{Eq:Eigval B:3}) to find
\begin{align}
 \label{Eq:Eigval B:4}
 \begin{split}
 \sum_{w \in V' \cap v^\perp} |B(\zero,w^\top G)|
 &= |\sett{(\Pi_1,\Pi_2) \in \mm \times \mm}{\Pi_1 \neq \Pi_2, \, \Pi_1 \cap \ell = \Pi_2 \cap \ell} \\
 &= \sum_{P \in \ell \setminus \set Q} |\sett{\Pi \in \mm}{P \in \Pi}|(|\sett{\Pi \in \mm}{P \in \Pi}|-1)\\
 & = \sum_{P \in \ell \setminus \set Q} |\sett{\Pi \in \mm}{P \in \Pi}|^2 - \underbrace{\sum_{P \in \ell \setminus \set Q} |\sett{\Pi \in \mm}{P \in \Pi}|}_{=|\mm| = d_R}.
 \end{split}
\end{align}
Now we plug (\ref{Eq:Eigval B:2}) and (\ref{Eq:Eigval B:4}) into (\ref{Eq:Eigval B:1}) to find that
\begin{align*}
 \chi_v(\Gamma) 
 &= q\left( \left(\sum_{P \in \ell \setminus \set Q} |\sett{\Pi \in \mm}{P \in \Pi}|^2 \right) - d_R \right) - d_R(d_R-q) \\
 &= q \left(\sum_{P \in \ell \setminus \set Q} |\sett{\Pi \in \mm}{P \in \Pi}|^2 \right) - d_R^2
\end{align*}
Using $|\ell \setminus \set Q| = q$ and $\sum_{P \in \ell \setminus \set Q} |\sett{\Pi \in \mm}{P \in \Pi}| = |\mm| = d_R$, we can rewrite this as
\[
 \chi_v(\Gamma) = q \sum_{P \in \ell \setminus \set Q} \left( |\sett{\Pi \in \mm}{P \in \Pi}| - \frac{|\mm|}q \right)^2.
\]
Note that for every line $\ell$ through $Q$, there are $q-1$ vectors $v \in V$ with $\ell = \vspan{P,v}$.
This shows that for every line $\ell$ through $Q$, we need to add the above eigenvalue to the spectrum of $M M^\top$ with multiplicity $q-1$.
We can then reconstruct the spectrum of $A_B$ from the spectrum of $M M^\top$ as described in the beginning of the proof, which finishes the argument.
\end{proof}

\Cref{Thm:Few or many} now follows directly from an application of the bipartite expander mixing lemma to the graph $B(C,i,\alpha)$.

\begin{proof}[Proof of \Cref{Thm:Few or many}]
Consider the graph $B = B(C,i,\alpha)$.
Define $L = \mf_{i,\alpha}$ and $R = C \setminus L$.
We know that $B$ is $(d_L,d_R)$-regular with $d_L = (q-1) d_R = (q-1)|\mm|$.
Note that $\frac {d_R}{|L|} = \frac{|\mm|}{q^{k-1}}$.
We know from \Cref{Lm:Biregular} (\ref{Lm:Biregular:Isomorphic}) that $B$ is isomorphic to $B(C,i,0)$, hence we can use \Cref{Prop:Eigval B} to deduce that the second largest eigenvalue $\lambda_2$ of $A_B$ equals $\sqrt \lambda$, with $\lambda$ as defined in the theorem statement.
Define $S = \mf \cap L$, and $T = \mf \cap R$.
Since $\mf$ is a coclique in $B$, the bipartite expander mixing lemma (\Cref{Res:Bip EML}) tells us that
\begin{equation}
 \label{Eq:Few or many:1}
 \left( 0 - \frac{|\mm|}{q^{k-1}} |S| |T| \right)^2 \leq \lambda_2^2 |S| \frac{|L|-|S|}{|L|} |T| \frac{|R|-|T|}{|R|} \leq \lambda |S| |T|.
\end{equation}
This is easily seen to imply the inequality from the theorem statement.
\end{proof}

\begin{rmk}
 \label{Rmk:HM:1}
 We can rephrase intersecting families of Hilton-Milner type to the language of the graph $B(C,i,\alpha)$.
 Such a family consists of a vertex $c \in R$, together with the vertices in $L$ that are not adjacent with $c$.
 Therefore, the size of a Hilton-Milner type family equals $|L| - d_R + 1 = q^{k-1} - |\mm| + 1$.
\end{rmk}

This simple observation tells us how large the intersection of an intersecting family $\mf$ and a star $\mf_{i,\alpha}$ needs to be to guarantee that $\mf$ is contained in $\mf_{i,\alpha}$.

\begin{lm}
 \label{Lm:Many is all}
 Suppose that $C$ is a linear $[n,k]_q$ code with projective system $\ms$.
 Let $\mm$ denote the set of hyperplanes that don't intersect $\mm$.
 If $\mf$ is an intersecting family, and $\mf_{i,\alpha}$ is a star, then either $\mf \subseteq \mf_{i,\alpha}$, or $|\mf \cap \mf_{i,\alpha}| \leq q^{k-1} - |\mm|$.
\end{lm}

\begin{proof}
 Consider the graph $B(C,i,\alpha)$.
 As reasoned above, $\mf$ is a coclique in this graph.
 Suppose that $\mf$ contains some codeword $c \notin \mf_{i,\alpha}$.
 We know from the proof of \Cref{Prop:Eigval B} that $c$ has degree $d_R = |\mm|$ in $B(C,i,\alpha)$.
 Therefore, $\mf$ contains at most $|\mf_{i,\alpha}| - d_R = q^{k-1} - |\mm|$ elements of $\mf_{i,\alpha}$.
\end{proof}

\subsection{Intersecting some star in more than few codewords}
 \label{Sec:More than few}

In this subsection, we explain, given a large intersecting family $\mf$, how to find a star $\mf_{i,\alpha}$ that has a sizeable intersection with $\mf$.

\begin{thm} 
 \label{Thm:more than few}
 Let $C$ be a linear $[n,k]_q$ code with projective system $\ms$.
 Suppose that through every point $P \notin \ms$ there are at least $t > 0$ hyperplanes that don't intersect $\ms$.
 If $\mf \subset C$ is an intersecting family, then there exists a star $\mf_{i,\alpha} = \sett{c \in C}{c(i) = \alpha}$ such that
 \[
  |\mf \cap \mf_{i,\alpha}| \geq \frac{|\mf|}q + \frac{q^{k-1}}n\left( \frac{|\mf|}{q^k} \left( \frac{|\mm|}t - 1 \right) + 1 - \frac{|\mm|}{qt} \right)
 \]
\end{thm}

\begin{proof}
 Let $\Gamma_0(C) = \Cay(C,\sett{c\in C}{\wt(c)=n})$ denote the graph on $C$ where codewords are adjacent if and only if they don't intersect.
 Let $A = A_{\Gamma_0(C)}$ denote its adjacency matrix.
 Let $\mm$ denote the set of hyperplanes that don't intersect $\ms$.
 We know from \Cref{Prop:Eigval Gamma_T} that the smallest eigenvalue of $A$ is $-|\mm|$, and the corresponding eigenspace $E_{-|\mm|}$ is spanned by the vectors $\xi_{v,\alpha} - \frac 1q \one$ with $\vspan v$ the points of $\ms$.
 Note that this means that $\xi_{v,\alpha}$ are the characteristic vectors of the stars.
 The second smallest eigenvalue (not counted with multiplicity) must be $\lambda = q t - |\mm|$.

 Consider the characteristic vector $\chi_\mf$ of $\mf$.
 We decompose it according to the eigenspaces of $A$.
 This yields that
 \[
  \chi_\mf = \frac{|\mf|}{q^k} \one + u_1 + u_2,
 \]
 with $u_1 \in E_{-|\mm|}$ and $u_2 \in \vspan{\one, E_{-|\mm|}}^\perp$.
 Note that $\one$ is an eigenvector of $A$ with as eigenvalue the degree of $\Gamma_0(C)$, which equals $(q-1)|\mm|$.
 We have
 \[
  |\mf| = \chi_\mf^\top \chi_\mf = \frac{|\mf|^2}{q^{2k}} \| \one \|^2 + \| u_1 \|^2 + \| u_2 \|^2
  = \frac{|\mf|^2}{q^{k}} + \| u_1 \|^2 + \| u_2 \|^2
 \]
 We also know that
 \begin{align*}
  0 &= \chi_\mf^\top A \chi_\mf = \left( \frac{|\mf|}{q^k} \one + u_1 + u_2 \right)^\top A \left( \frac{|\mf|}{q^k} \one + u_1 + u_2 \right) \\
  &= \frac{|\mf|^2}{q^{2k}} (q-1)|\mm| \| \one \|^2 - |\mm| \| u_1 \|^2 + u_2^\top A u_2.
 \end{align*}
 In addition,
 \[
  u_2^\top A u_2 \geq \lambda \| u_2 \|^2
  = (qt - |\mm|)\left( |\mf| - \frac{|\mf|^2}{q^k} - \| u_1 \|^2 \right).
 \]
 Hence,
 \begin{align*}
  0 &\geq \frac{|\mf|^2}{q^{k}} (q-1)|\mm| - |\mm| \| u_1 \|^2 + (qt - |\mm|)\left( |\mf| - \frac{|\mf|^2}{q^k} - \| u_1 \|^2 \right) \\
  &= |\mf| \left( \frac{|\mf|}{q^{k-1}}(|\mm|-t) + qt-|\mm| \right) - q t \| u_1 \|^2,
 \end{align*}
 which implies that
 \[
  \| u_1 \|^2 \geq |\mf| \left( \frac{|\mf|}{q^k} \left( \frac{|\mm|}t - 1 \right) + 1 - \frac{|\mm|}{qt} \right).
 \]

 Now make a matrix $M$ whose rows are labelled by $C$, whose columns are labelled by $\{1,\dots,n\} \times \FF_q$, and where $M(c,(i,\alpha))$ equals 1 if $c(i) = \alpha$ and 0 otherwise.
 In other words, the columns of $M$ are the characteristic vectors of the stars.
 We determine the eigenvalues and eigenspace of $M M^\top$.
 First, note that
 \[
  M M^\top \one = M (q^{k-1} \one) = n q^{k-1} \one.
 \]
 Next, consider the generator matrix $G$ of $C$ corresponding to $\ms$, and let $v$ be the $i$\th column of $G$, so that $\xi_{v,\alpha}$ is the characteristic vector of $\mf_{i,\alpha}$.
 Then
 \[
  M^\top \xi_{v,\alpha}(j,\beta) = \begin{cases}
  q^{k-1} & \text{if } (i,\alpha) = (j,\beta), \\
  0 & \text{if } i = j, \, \alpha \neq \beta, \\
  q^{k-2} & \text{if } i \neq j.
  \end{cases}
 \]
 Therefore,
 \[
  M M^\top \xi_{v,\alpha} (c) = \sum_{j=1}^n M^\top \xi_{v,\alpha} (j,c(j))
  = (n-1) q^{k-2} + \begin{cases}
   q^{k-1} & \text{if } c(i) = \alpha, \\
   0 & \text{otherwise}.
  \end{cases}
 \]
 In other words, $M M^\top \xi_{v, \alpha} = q^{k-1} \xi_{v,\alpha} + (n-1) q^{k-2} \one$.
 It follows that
 \[
  M M^\top \left(\xi_{v,\alpha} - \frac 1q \one \right) = q^{k-1} \xi_{v,\alpha} + (n-1) q^{k-2} \one - \frac 1q n q^{k-1} \one = q^{k-1} \left(\xi_{v,\alpha} - \frac 1q \one \right).
 \]
 Since the vectors $\xi_{v,\alpha} - \frac 1q \one$ span $E_{-|\mm|}$, every vector in $E_{-|\mm|}$ is an eigenvector of $M M^\top$ with eigenvalue $q^{k-1}$, and this applies in particular to $u_1$.
 Lastly, since the column space of $M$ spans $\vspan{\one, E_{-|\mm|}}$, every vector orthogonal to $\vspan{\one, E_{-|\mm|}}$ is in the kernel of $M M^\top$, and this applies in particular to $u_2$.
 We conclude that
 \begin{align*}
  \chi_{\mf}^\top M M^\top \chi_{\mf}
  &= \left( \frac{|\mf|}{q^k} \one + u_1 + u_2 \right)^\top M M^\top \left( \frac{|\mf|}{q^k} \one + u_1 + u_2 \right) \\
  &= \left( \frac{|\mf|}{q^k} \one + u_1 + u_2 \right)^\top \left( n q^{k-1} \frac{|\mf|}{q^k} \one + q^{k-1} u_1 \right)
  = n q^{k-1} \frac{|\mf|^2}{q^{2k}} \| \one \|^2 + q^{k-1} \| u_1 \|^2 \\
  &\geq n \frac{|\mf|^2}q + q^{k-1} |\mf| \left( \frac{|\mf|}{q^k} \left( \frac{|\mm|}t - 1 \right) + 1 - \frac{|\mm|}{qt} \right)
 \end{align*}
On the other hand, if we denote the star $\sett{c \in C}{c(i) = \alpha}$ by $\mf_{i,\alpha}$, then $\chi_\mf^\top M(i,\alpha) = |\mf \cap \mf_{i,\alpha}|$.
Thus,
\[
 \chi_{\mf}^\top M M^\top \chi_{\mf}
 = \sum_{i=1}^n \sum_{\alpha \in \FF_q} |\mf \cap \mf_{i,\alpha}|^2
 = |\sett{(i,\alpha,c_1,c_2) \in \{1,\dots,n\} \times \FF_q \times C \times C}{c_1(i) = c_2(i) = \alpha}|.
\]
By a simple averaging argument, there must exist a $c_1 \in C$ such that
\begin{align*}
 |\sett{(i,c_2) \in \{1,\dots,n\} \times C}{c_2(i) = c_1(i)}| 
 &\geq \frac{1}{|\mf|} \chi_{\mf}^\top M M^\top \chi_{\mf} \\
 &\geq n \frac{|\mf|}q + q^{k-1}\left( \frac{|\mf|}{q^k} \left( \frac{|\mm|}t - 1 \right) + 1 - \frac{|\mm|}{qt} \right)
\end{align*}
Again by a simple averaging argument, there must exist an $i \in \{1, \dots, n\}$ such that
\begin{align*}
 |\mf \cap \mf_{i,c_1(i)}|
 & \geq \frac 1n |\sett{(i,c_2) \in \{1,\dots,n\} \times C}{c_2(i) = c_1(i)}| \\
 &\geq \frac{|\mf|}q + \frac{q^{k-1}}n\left( \frac{|\mf|}{q^k} \left( \frac{|\mm|}t - 1 \right) + 1 - \frac{|\mm|}{qt} \right).
 \qedhere
\end{align*}
\end{proof}

If $|\mf| = q^{k-1}$, then we can deduce from the EKR module property that $u_2 = \zero$, and the bound of \Cref{Thm:more than few} does not depend on the value of $t$.

\begin{crl}
 \label{Crl:More than few}
 Let $C$ be a linear $[n,k]_q$ code that has the EKR module property.
 If $\mf$ is an intersecting family in $C$ of size $q^{k-1}$, then there exists a star $\mf_{i,\alpha}$ such that
 \[
  |\mf \cap \mf_{i,\alpha}| \geq q^{k-2} \left( 1 + \frac{q-1}n \right).
 \]
\end{crl}

\begin{proof}
 Let $\mm$ be the set of hyperplanes that don't intersect the projective system $\ms$ associated to $C$.
 Since the EKR module property holds, there exists an integer $t > 0$ such that every point $P \notin \ms$ is incident with at least $t$ hyperplanes of $\mm$.
 By \Cref{Thm:more than few}, there exists a star $\mf_{i,\alpha}$ such that
 \begin{align*}
  |\mf \cap \mf_{i,\alpha}| 
  &\geq \frac{|\mf|}q + \frac{q^{k-1}}n\left( \frac{|\mf|}{q^k} \left( \frac{|\mm|}t - 1 \right) + 1 - \frac{|\mm|}{qt} \right) \\
  &= q^{k-2} + \frac{q^{k-1}}n \left( \frac 1q \left( \frac{|\mm|}t - 1 \right) + 1 - \frac{|\mm|}{qt} \right)
  = q^{k-2} + \frac{q^{k-1}}n\left( 1 - \frac 1q \right),
 \end{align*}
 which is equivalent with the corollary statement.
\end{proof}

\subsection{The strict EKR property}
 \label{Sec:Strict sub}

In this section, we will prove \Cref{Thm:Codes} (\ref{Thm:Codes:3}), which is stated again below as \Cref{Thm:Strict EKR}, using the tools developed thus far in this section.
The theorem provides us with a sufficient condition for the strict EKR property to hold.
This condition is particularly intersecting when the length of the code is subquadratic in the field order.

\begin{thm}
 \label{Thm:Strict EKR}
 Let $C$ be a linear $[n,k]_q$ code with projective system $\ms$.
 Suppose that no 3 points of $\ms$ are collinear\footnote{This is equivalent to the dual code $C^\perp$ of $C$ having minimum weight at least 3.}.
 Let $\mm$ denote the set of hyperplanes that don't intersect $\ms$.
 Suppose that for every point $P \notin \ms$,
 \[
  \left| |\sett{\Pi \in \mm}{P \in \Pi}| - \frac{|\mm|}q  \right| < \frac{|\mm|}{q \min \{ \sqrt n, q-1 \}}.
 \]
 Then the only intersecting families in $C$ of size $q^{k-1}$ are the stars.
\end{thm}

\begin{proof}
 Let $\mf$ be an intersecting family of size $q^{k-1}$ in $C$.
 By \Cref{Crl:More than few}, there exists a star $\mf_{i,\alpha}$ such that
 \begin{equation}
  \label{Eq:Strict EKR:1}
  |\mf \cap \mf_{i,\alpha}| \geq q^{k-2} \left( 1 + \frac{q-1}n \right).
 \end{equation}
 On the other, hand we can apply \Cref{Thm:Few or many} to say something about the possible values of $|\mf \cap \mf_{i,\alpha}|$.
 Instead of applying the statement of \Cref{Thm:Few or many}, we instead apply the first inequality of (\ref{Eq:Few or many:1}):
 Let $B = B(C,i,\alpha)$, $L = \mf_{i,\alpha}$, $R = C \setminus L$, $S = \mf \cap L$, and $T = \mf \cap R$.
 Note that $|L| = |\mf| = q^{k-1}$ and $|R| = (q-1)q^{k-1}$.
 Note also that $|S| + |T| = |\mf| = q^{k-1}$.
 Define $\lambda$ as in \Cref{Thm:Few or many}.
 Then (\ref{Eq:Few or many:1}) tells us that
 \[
  \left( \frac{|\mm|}{q^{k-1}} |S| (q^{k-1}-|S|) \right)^2 \leq \lambda |S| \frac{q^{k-1}-|S|}{q^{k-1}}(q^{k-1}-|S|)\frac{(q-2)q^{k-1} + |S|}{(q-1)q^{k-1}}.
 \]
 From now on, we will assume that $\mf \neq \mf_{i,\alpha}$, and derive a contradiction.
 In this case, $|S| < q^{k-1}$, and the above inequality simplifies to
 \[
  |\mm|^2 |S| \leq \lambda \frac{(q-2)q^{k-1} + |S|}{q-1}.
 \]
 This is equivalent to
 \[
  \left( |\mm|^2 \frac{q-1}\lambda - 1 \right) |S| \leq (q-2) q^{k-1}.
 \]
 Using (\ref{Eq:Strict EKR:1}), this implies that
 \[
  \left( |\mm|^2 \frac{q-1}\lambda - 1 \right) q^{k-2} \left( 1 + \frac{q-1}n \right) \leq (q-2) q^{k-1}.
 \]
 This inequality can be rewritten as
 \begin{equation}
  \label{Eq:Strict EKR:2}
  \lambda \geq |\mm|^2 \frac{q-1}{\frac{q(q-2)}{1 + \frac{q-1}n}+1}.
 \end{equation}
 On the other hand, define $r$ to be the maximum value of
 \[
  \left| |\sett{\Pi \in \mm}{P \in \Pi}| - \frac{|\mm|}q  \right|
 \]
 over all points $P \notin \ms$.
 Recall that
 \[
  \lambda = q \sum_{P \in \ell \setminus \{Q\}} \left( |\sett{\Pi \in \mm}{P \in \Pi}| - \frac{|\mm|}q  \right)^2 
 \]
 for some point $Q \in \ms$ and some line $\ell$ through $Q$.
 By the assumptions of the theorem, $\ell \setminus \{Q\}$ contains at most one point of $\ms$, which is of course not incident to any of the hyperplanes of $|\mm|$.
 This yields that
 \[
  \lambda \leq q (q-1)r^2 + \frac{|\mm|^2}{q}.
 \]
 Combined with (\ref{Eq:Strict EKR:2}), this implies that
 \[
  q(q-1)r^2 + \frac{|\mm|^2}q \geq |\mm|^2 \frac{q-1}{\frac{q(q-2)}{1 + \frac{q-1}n}+1}
 \]
 This is equivalent to the following inequality (we omit some tedious arithmetic)
 \[
  r^2 \geq \frac{|\mm|^2}{q(q-1)} \left( \frac{q-1}{\frac{q(q-2)}{1 + \frac{q-1}n}+1} - \frac 1q \right)
  = \frac{|\mm|^2}{q^2(q-1)} \frac{n+q^2-q-1}{n(q-1)+1}
 \]
 On the other hand, by the assumptions from the theorem,
 \[
  r^2 < \frac{|\mm|^2}{q^2} \frac 1{\min\{n,(q-1)^2\}}.
 \]
 This implies that
 \begin{align}
  \label{Eq:Strict EKR:3}
  && \frac1{q-1} \frac{n+q^2-q-1}{n(q-1)+1} 
  & \leq r^2 \frac{q^2}{|\mm|^2} 
  < \frac{1}{\min\{n,(q-1)^2\}} \notag \\ \implies &&
  \min\{n,(q-1)^2\}(n+q^2-q-1) 
  & < (q-1)(n(q-1)+1).
 \end{align}
 We will prove that this leads to a contradiction.
 First, suppose that $n < (q-1)^2$.
 Then (\ref{Eq:Strict EKR:3}) becomes
 \[
  n \left( n + (q^2 - q - 1) \right) < n(q-1)^2 + q-1,
 \]
 which simplifies to
 \[
  n^2 + (q-2)n < q-1.
 \]
 This is a contradiction for $n > 1$ and $q \geq 2$.

 Next, suppose that $n \geq (q-1)^2$.
 Then (\ref{Eq:Strict EKR:3}) becomes
 \[
  (q-1)^2(n+q^2-q-1) < (q-1)(n(q-1) + 1),
 \]
 which after dividing by $q-1$ becomes
 \[
  (q-1)(q^2-q-1) < 1.
 \]
 Using that $q \geq 2$, this also yields a contradiction.
\end{proof}

\subsection{A stability result}
 \label{Sec:Stab}

In this subsection, we go beyond the strict EKR property, and use the developed tools to prove \Cref{Thm:Stability}, which describes a stability result.

\bigskip

\noindent \textbf{\Cref{Thm:Stability}.}
\textit{Let $\mq$ be an infinite set of prime powers, and suppose that for each $q \in \mq$, there exists a linear code $C_q$ with parameters $[n_q,k]_q$ (where $n_q$ can vary, but $k$ is fixed).  Denote by $\ms_q$ the projective system associated with $C_q$, and by $\mm_q$ the set of hyperplanes that don't intersect $\ms_q$.  Suppose that there exist constants $a,b,\delta,\mu,\tau$ with $0 < \mu < 1$ such that the following properties hold for all of $q \in \mq$:  \begin{enumerate}   \item no 3 points of $\ms_q$ are collinear,   \item $n_q \leq aq+b$,   \item $\big| |\mm_q| - \mu q^{k-1}\big| \leq \tau q^{k-2}$,   \item for each point $P$ of $\pg(k-1,q)$ outside of $\ms_q$, $\left| \sett{\Pi \in \mm_q}{P \in \Pi} - \frac{|\mm_q|}q \right| \leq \delta q^{k-3}$.  \end{enumerate}  Then there exists a constant $\sigma$ such that in each of the codes $C_q$, every intersecting family $\mf$ with \[  |\mf| \geq \max\left\{ \frac1{\sqrt{1+a^{-1}}}, \, 1-\mu \right\} q^{k-1} + \sigma q^{k-2}  \] is contained in a star.}

\begin{proof}
 Suppose that the family of linear codes $(C_q)_{q \in \mq}$ satisfies the hypotheses of the theorem, and let $r$ denote $\max \Big\{ \frac1{\sqrt{1+a^{-1}}}, \, 1-\mu \Big\} $.
 We will prove that the theorem holds for sufficiently large values of $q$, which suffices, since we can take $\sigma$ large enough so that $r q^{k-1} + \sigma q^{k-2} > q^k$ for small values of $q$.
 Let $\mf_{i,\alpha}$ be the star that has the biggest intersection with $\mf$, and write $x = |\mf \cap \mf_{i,\alpha}|$.
 
\bigskip

\noindent {\it Claim 1: There exists a constant $\tau_1$ (independent of $q$) such that for $q$ sufficiently large,
 \[
  x \geq \frac{1 + a^{-1}}q |\mf| - \tau_1 q^{k-3}.
 \]}
We apply \Cref{Thm:more than few}.
Let $t$ be the minimum number of hyperplanes of $\mm_q$ incident with a point $P \notin \ms_q$.
Then $t \geq \frac{|\mm_q|}q - \delta q^{k-3} \geq \mu q^{k-2} - (\tau + \delta) q^{k-3}$.
Hence, $t > 0$ for $q$ sufficiently large.
By \Cref{Thm:more than few},
\begin{equation}
 \label{Eq:Stab 1}
 x \geq \frac{|\mf|}q + \frac{q^{k-1}}{aq+b} \left( \frac{|\mf|}{q^k} \left( \frac{|\mm_q|}t - 1\right) + 1 - \frac{|\mm_q|}{qt} \right).
\end{equation}
Since $t > 0$, the weak EKR property holds, hence $|\mf| \leq q^{k-1}$.
Thus, the right hand side is decreasing in $\frac{|\mm_q|}t$.
It holds that
\[
  \frac{|\mm_q|}t \leq \frac{|\mm_q|}{\frac{|\mm_q|}q - \delta q^{k-3}} = \frac q{1 - \frac{\delta q^{k-2}}{|\mm_q|}}
 \leq \frac q{1 - \frac{\delta}{\mu q - \tau}} 
\]
Therefore, there exists a constant $\tau_1'$ such that $\frac{|\mm_q|}t \leq q + \tau_1'$ for $q$ sufficiently large.
Plugging this into (\ref{Eq:Stab 1}) tells us that
\[
 x \geq \frac{|\mf|}q \left( 1 + \frac{q + \tau_1' - 1}{aq+b} \right) - \frac{q^{k-1}}{aq+b} \frac{\tau_1'}q.
\]
Since the right hand side is decreasing in $\tau'_1$, we may suppose that $\tau_1' \geq \frac{b-1}a$, so that $\frac{q+\tau_1' - 1}{aq+b} \geq \frac 1 a$.
It is clear that there exists a constant $\tau_1$ such that $\frac{q^{k-1}}{aq+b} \frac{\tau_1'}q  \leq \tau_1 q^{k-3}$ for $q$ sufficiently large.
This proves the claim.

\bigskip

\noindent {\it Claim 2: There exists a constant $\tau_2$ (independent of $q$) such that for $q$ sufficiently large,
\begin{equation} \label{Eq:Stab 2}
 x(|\mf|-x) \leq q^{2k-3} + \tau_2 q^{2k-4}.
\end{equation}}
We apply \Cref{Thm:Few or many}.
Let $Q \in \ms_q$ be the point corresponding to the star $\mf_{i,\alpha}$, and for each line $\ell$ through $Q$, define
\[
 \lambda_\ell = q \sum_{P \in \ell \setminus \set Q} \left( |\sett{\Pi \in \mm_q}{P \in \Pi} - \frac{|\mm_q|}q \right)^2.
\]
Then we have that $\lambda_\ell \leq q \cdot q \left(\delta q^{k-3}\right)^2$ if $\ell$ contains no second point of $\ms_q$, and 
\[
 \lambda_\ell \leq q \cdot \left( \left( \frac{|\mm_q|}q \right)^2 + (q-1)\left(\delta q^{k-3}\right)^2 \right)
\]
if $\ell$ does contain a second point of $\ms_q$.
For $q$ sufficiently large, the second upper bound dominates.
By \Cref{Thm:Few or many},
\[
 x (|\mf|-x) \leq q \cdot \left( \left( \frac{|\mm_q|}q \right)^2 + (q-1)\left(\delta q^{k-3}\right)^2 \right) \left( \frac{q^{k-1}}{|\mm_q|} \right)^2 = q^{2k-3} + (q-1) \delta^2 q^{2k-5} \left( \frac{q^{k-1}}{|\mm_q|} \right)^2.
\]
Since $|\mm_q| \geq \mu q^{k-1} - \tau q^{k-2}$, and $\mu > 0$, there exists a constant $\tau_2'$ such that $\frac{q^{k-1}}{|\mm_q|} \leq \tau_2'$ for $q$ sufficiently large.
Then the claim follows by taking $\tau_2 = \tau_2'^2 \delta^2$.

\bigskip

We can now finish the proof.
By \Cref{Lm:Many is all}, it suffices to prove that $x > q^{k-1} - |\mm|$, hence it is enough to prove that $x > x_2 = (1-\mu) q^{k-1} + \tau q^{k-2}$.
Define $x_1 = \frac{1+a^{-1}}q |\mf| - \tau_1 q^{k-3}$.
If we can prove that the inequality (\ref{Eq:Stab 2}) does not hold for $x_1$ and $x_2$, then by the theory of quadratic polynomials over the reals, it does not hold on the interval $[x_1,x_2]$.
Since $x \geq x_1$ by Claim 1, this then implies the desired inequality $x > x_2$.

First consider $x = x_1$.
Plugging in $|\mf| \geq \frac1{\sqrt{1+a^{-1}}} q^{k-1} + \sigma q^{k-2}$, we find
\begin{align*}
 x_1 (|\mf|-x_1) \geq & \left( \sqrt{1+a^{-1}} q^{k-2} + (1+a^{-1})(\sigma - \tau_1) q^{k-3} \right) \\
 & \cdot \left( \frac1{\sqrt{1+a^{-1}}} q^{k-1} + \left(\sigma -  \sqrt{1+a^{-1}} \right) q^{k-2} - (1+a^{-1})(\sigma - \tau_1) q^{k-3} \right)
\end{align*}
The leading term on the right hand side is $q^{2k-3}$, and the coefficient of $q^{2k-4}$ can be made arbitrarily large by enlarging $\sigma$.
Hence, for some choice of the constant $\sigma$, this will ensure that inequality (\ref{Eq:Stab 2}) fails for $x_1$ if $q$ is sufficiently large.

Next, consider $x = x_2$.
Then using that $|\mf| \geq (1-\mu) q^{k-1} + \sigma q^{k-2}$, we find that
\[
 x_2 (|\mf|-x_2) \geq \left( (1-\mu) q^{k-1} + \tau q^{k-2} \right) (\sigma - \tau) q^{k-2} > (1-\mu) (\sigma - \tau) q^{2k-3},
\]
where the last inequality holds under the assumption $\sigma \geq \tau$, which we may indeed assume.
We can choose $\sigma$ sufficiently large such that $(1-\mu) (\sigma - \tau) > 1$.
Then it immediately follows that $x_2 (|\mf|-x_2) > q^{2k-3} + \tau_2 q^{2k-4}$ for $q$ sufficiently large.
This finishes the proof.
\end{proof}

\begin{rmk}
 \label{Rmk:HM:2}
 Note that if $\max\left\{ \frac1{\sqrt{1+a^{-1}}}, \, 1-\mu \right\}  = 1-\mu$, then this constant is optimal in the above theorem due to the existence of Hilton-Milner type families, whose size equals $q^{k-1} - |\mm| + 1 = (1-\mu)q^{k-1} + \mo_q(q^{k-2})$ by \Cref{Rmk:HM:1}.
It is not clear to the author whether there exists a family of linear codes for which $1-\mu > \frac1{\sqrt{1+a^{-1}}}$.
\end{rmk}

\section{Intersection problems for polynomials}
 \label{Sec:Stab pol}

 In this section, we prove the theorems concerning ($t$-)intersecting families of polynomials stated in the introduction.
 We translate the problem to the setting of linear codes, so that we can apply the previously developed tools.
 This makes the following observation crucial, despite its trivial nature.

 \begin{lm}
  \label{Lm:Trivial but crucial}
  Suppose that $q$ is a prime power, and $k < q$.
  Recall the map $Ev$ from \Cref{Df:RS}.
  Consider a family $\mf \subseteq \hom k$.
  Define $\me = \sett{Ev(f(X,1))}{f \in \mf}$.
  \begin{enumerate}
   \item $\mf$ is a $t$-intersecting in $\hom k$ if and only if $\me$ is a $t$-intersecting family in $ERS(q,k)$.
   \item $\mf$ is a $t$-star in $\hom k$ if and only if $\me$ is a $t$-star in $ERS(q,k)$. 
  \end{enumerate}
 \end{lm}

 \begin{proof}
  This follows immediately from the definition of $t$-intersecting families in $\hom k$ (\Cref{Df:Pol int}), $t$-intersecting families in linear codes (\Cref{Df:Code int}), the implied definition of $t$-stars in both settings, and the definition of $Ev$ and $ERS(q,k)$ (\Cref{Df:RS}).
 \end{proof}

\subsection{A general bound for \texorpdfstring{$t$}{t}-intersecting families}

\noindent \textbf{\Cref{Prop:t int}.}
\textit{Suppose that $t \leq k < q$, and $k \geq 2$.  Let $\mf$ be a $t$-intersecting family in $\hom k$. \begin{enumerate}  \item If $t < k$, then $|\mf| \leq q^{k+1-t}$.  \item If $t = k$, then $|\mf| < \frac{q^2-1}k + 1$. \end{enumerate}}

\begin{proof}
 (1) We will prove (1) for $t$-intersecting families $\mf$ in $ERS(q,k)$, which is equivalent by \Cref{Lm:Trivial but crucial}.
 Using \Cref{Prop:Weak EKR:t}, it suffices to show that if $t < k$, then $ERS(q,k)$ has a linear $[q+1,t]_q$ MDS subcode $C'$.
 It is well known that $\fq[X]$ contains irreducible polynomials of every degree (this follows e.g.\ from the fact that $\fq$ has a field extension of every finite order).
 Take an irreducible polynomial $f \in \fq[X]$ of degree $k+1-t \geq 2$.
 Define $C'$ to be the image of $\sett{f\cdot g}{g \in \fq[X]_{\leq t-1}}$ under $Ev$.
 Then
 \[
  Ev(f \cdot g) = (f(x_1) g(x_1), \dots, f(x_q) g(x_q), f(\infty_{k+1-t}) g(\infty_{t-1})).
 \]
 Since $f$ has no roots on $\fq \cup \{\infty_{k+1-t}\}$, $C'$ is equivalent to $ERS(q,t-1)$, hence it is a linear $[q+1,t]_q$ MDS code.

 \bigskip
 
 (2) The notation will be easier if we represent the polynomials of $\hom k$ as the polynomials $\fq[X]_{\leq k}$ evaluated on $\fqi$.
 Recall that this transforms a polynomial $F \in \hom k$ into the polynomial $f(X) = F(X,1)$, and we define $f(\infty) = F(1,0)$.

 Take a polynomial $f(X) \in \mf$, and consider the set
 \[
  V = \sett{(x,g) \in (\fqi) \times \mf}{g(x) = f(x)}.
 \]
 We perform a double count on $V$.
 Given a polynomial $g \in \mf$, the number of values $x \in \fqi$ with $f(x) = g(x)$ equals $q+1$ if $f=g$ and $k$ otherwise.
 Therefore,
 \begin{align}
  \label{Eq:t int:1}
  |V| = (q+1) + k(|\mf|-1).
 \end{align}
 On the other hand, we can take a point $x \in \fqi$, and let $y$ be $f(x)$.
 Without loss of generality, we may assume that $x = \infty$ and $y=0$.
 Then $\sett{g \in \mf}{g(\infty) = 0}$  constitutes a $(k-1)$-intersecting family in $\fq[X]_{\leq k-1}$ (ignoring the evaluation at $\infty$).
 Therefore, by \Cref{Res:Weak EKR}, this set contains at most $q$ polynomials.
 We conclude that
 \begin{align}
  \label{Eq:t int:2}
  |V| \leq (q+1) q.
 \end{align}
 If we combine (\ref{Eq:t int:1}) and (\ref{Eq:t int:2}), we find that
 \[
  |\mf| = \frac{|V| - (q+1)}k + 1 \leq \frac{q^2-1}k + 1.
 \]
 If equality holds, then every point $(x,y) \in (\fqi)\times\fq$ is incident with either $0$ or $q$ polynomials of $\mf$.
 Since each polynomial of $\mf$ is incident with exactly $q+1$ points, $q$ must then divide $(q+1)\left( \frac{q^2-1}k + 1 \right)$.
 This implies that $k \equiv 1 \pmod q$.
 Since $2 \leq k < q$, this yields a contradiction.
\end{proof}

\begin{rmk}
 If we take $t=k=1$, then $\mf = \hom 1$ is an intersecting family, and the equality
 \[
  |\mf| = q^2 = \frac{q^2-1}k + 1
 \]
 holds.
\end{rmk}

\subsection{A stability result for intersecting families}

We are ready to prove \Cref{Thm:Main}.
We prove it by applying \Cref{Thm:Stability} to the extended Reed-Solomon codes.
In order to do so, we need some properties concerning their projective systems, the normal rational curves.
The reader should recall the function $\nu$ defined in \Cref{Df:nu}.

We start by reviewing some properties of normal rational curves.
For starters, every projectivity of $\pg(1,q)$ can be lifted to a projectivity of $\pg(k,q)$ that stabilises $\mn$.

\begin{res}[{\cite[Theorem 6.30 (vi)]{Hirschfeld:Thas}}]
 \label{Res:Lift}
 For every map $\phi \in \PGL(2,q)$, there exists a map $\Phi \in \PGL(k+1,q)$ such that for all points $P$ of $\pg(1,q)$, $\nu_k(\phi(P)) = \Phi(\nu_k(P))$.
 As a corollary, the stabiliser in $\PGL(k+1,q)$ of a normal rational curve $\mn$ in $\pg(k,q)$ acts triply transitively on the points of $\mn$.
\end{res}

We want to use the tools developed in the previous section to characterise large intersecting families in extended Reed-Solomon codes.
In order to do so, we need to prove that if $\mn$ is a normal rational curve in $\pg(k,q)$ and $P$ is a point outside of $\mn$, then the number of hyperplanes through $P$ that don't intersect $\mn$ cannot deviate too much from a proportion $1/q$ of these hyperplanes.
We start by establishing the number of hyperplanes intersecting $\mn$ in a given number of points.

\begin{df}
 \label{Df:mu}
 For integers $t$ and $k$ with $0 \leq t \leq k$, define the constants
 \[
  \mu_{k,t} = \frac1{t!}\sum_{j=0}^{k-t} \frac{(-1)^j}{j!}.
 \]
\end{df}

\begin{lm}
 \label{Crl:Weight dist}
 Consider a normal rational curve $\mn$ in $\pg(k,q)$, $k < q$.
 Suppose that $t \leq k$.
 Then there are
\[
 \binom{q+1}{t} \sum_{j=0}^{k-t} (-1)^j \binom{q+2-t}j q^{k-t-j} = \mu_{k,t} q^k + \mo_q(q^{k-1})
\]
 hyperplanes that intersect $\mn$ in exactly $t$ points.
\end{lm}

\begin{proof}
 Recall that every non-zero codeword $c$ of $ERS(q,k)$ of weight $n-t$ corresponds to a unique hyperplane of $\pg(k,q)$ intersecting $\mn$ in $t$ points, and that vice versa, each hyperplane corresponds to $q-1$ codewords.
 The statement of the lemma then directly follows from \Cref{Res:Wt Dist MDS}.
\end{proof}

Next, we check that the constants $\mu_{k,t}$ sum to one.
This can be checked directly by computation, but we can also avoid having to do any arithmetic.

\begin{lm}
 \label{Lm:Sum mu}
 For every positive integer $k$,
 \[
  \sum_{t=0}^k \mu_{k,t} = 1.
 \]
\end{lm}

\begin{proof}
 Fix $k$.
 Consider a normal rational curve $\mn$ in $\pg(k,q)$ and $q > k$.
 If we fix $t$, the number of hyperplanes that intersect $\mn$ in exactly $t$ points is a polynomial $F_{k,t}(q)$ in $q$ of degree at most $k$ by \Cref{Crl:Weight dist}, and its $q^k$ coefficient equals $\mu_{k,t}$.
 The total number of hyperplanes equals $G_k(q) = q^k + q^{k-1} + \ldots + 1$.
 We find that $\sum_{t=0}^k F_{k,t}(q) = G_k(q)$ for all prime powers $q > k$, hence the polynomials must coincide.
 Looking at the coefficient of $q^k$, this proves exactly the statement of the lemma.
\end{proof}

Consider a normal rational curve $\mn$ in $\pg(k,q)$, and a point $Q \in \mn$.
We call $\mn^\circ = \mn \setminus \{Q\}$ a \emph{punctured normal rational curve}, and $Q$ the \emph{punctured point} of $\mn^\circ$.

\begin{res}[{\cite[Lemma 6.31 (i)]{Hirschfeld:Thas}}]
 Suppose that $\mn$ is a normal rational curve in $\pg(k,q)$ and take a point $Q \in \mn$.
 Then $Q$ projects the remaining points of $\mn$ onto a punctured normal rational curve in $\pg(k-1,q)$.
\end{res}


Given a point $Q$ on the normal rational curve $\mn$, let $\mn^\circ$ denote the punctured normal rational curve equal to the projection of $\mn \setminus \{Q\}$ from $Q$, and let $Q'$ be the punctured point.
Then $\proj_Q^{-1}(Q')$ is a line $\ell$ through $Q$ in $\pg(k,q)$.
This line is called the \emph{tangent line} of $Q$ to $\mn$.
Call a line intersecting $\mn$ in two points \emph{bisecant} to $\mn$.

\begin{res}[{\cite[Lemma 6.31]{Hirschfeld:Thas}}]
 \label{Lm:One special line}
 Suppose that $\mn$ is a normal rational curve in $\pg(k,q)$ with $k \geq 3$ and let $P$ be a point not on $\mn$.
 Then there is at most one line through $P$ that is either a tangent or bisecant line to $\mn$.
\end{res}


For the next proposition, we will use the constants $\mu_{k,t}$ defined in \Cref{Df:mu}.

\begin{prop}
 \label{Prop:Bounded deviation}
 For each integer $k \geq 3$, there exists a constant $\delta_k$ (independent of $q$) such that the following holds for every $q > k$.
 Let $\mn$ denote a normal rational curve in $\pg(k,q)$, and let $P$ be a point not on $\mn$.
 For $0 \leq t \leq k$, let $s_t(P)$ denote the number of hyperplanes incident with $P$ that intersect $\mn$ in exactly $t$ points.
 Then
 \[
  \left| s_t(P) - \mu_{k,t} q^{k-1} \right| \leq \delta_k q^{k-2}.
 \]
\end{prop}

\begin{proof}
 For $k = 3$, this follows from \cite[Tables 1 and 2]{Bartoli:DMP:20}.
 We prove that it holds for bigger values of $k$ by induction.
 Take a point $P$ of $\pg(k,q)$, not on $\mn$.
 For each point $Q \in \mn$, let $s_t(P,Q)$ denote the number of hyperplanes incident with both $P$ and $Q$ that intersect $\mn$ in exactly $t$ points.
 Then $\sum_{Q \in \mn} s_t(P,Q) = t s_t(P)$.

 First consider the case where $t > 0$, and take some point $Q \in \mn$.
 Consider the projection from $Q$, let $P'$ be $\proj_Q(P)$, and let $\mn^\circ$ be the image of $\mn \setminus \{Q\}$ under $\proj_Q$.
 Let $Q'$ be the punctured point of $\mn^\circ$.
 It follows from \Cref{Lm:One special line} that at most two points $Q \in \mn$ project $P$ onto a point of $\mn^\circ \cup \{Q'\}$.
 In the other cases, $s_t(P,Q)$ equals the number of hyperplanes in $\pg(k-1,q)$ through $\proj_Q(P)$ that intersect $\mn^\circ \cup \{Q'\}$ in exactly $t-1$ points.
 By the induction hypothesis,
 \[
 \left|  s_t(P,Q) - \mu_{k-1,t-1} q^{k-2} \right| \leq \delta_{k-1} q^{k-3}.
 \]
 Therefore,
 \begin{align*}
  \sum_{Q \in \mn} s_t(P,Q)  
 & \geq (q-1) \left(\mu_{k-1,t-1} q^{k-2} - \delta_{k-1} q^{k-3} \right), \text{ and}  \\
 \sum_{Q \in \mn} s_t(P,Q) 
 &\leq (q-1) \left( \mu_{k-1,t-1} q^{k-2} + \delta_{k-1} q^{k-3} \right) + 2 (q^{k-2} + \ldots + q + 1).
 \end{align*}
 Write $\delta' = \delta_{k-1} + 4$.
 Note that $\frac 1 t \mu_{k-1,t-1} = \mu_{k,t}$.
 Then using $s_t(P) = \frac 1t \sum_{Q \in \mn} s_t(P,Q)$, we find that
 \[
  \left| s_t(P) - \mu_{k,t} q^{k-1} \right| \leq \frac{\delta'}{t} q^{k-2}.
 \]

 Lastly, consider the case $t=0$.
 We have that
 \[
  s_0(P) = q^{k-1} + \ldots + q + 1 - \sum_{t=1}^k s_t(P).
 \]
 Using \Cref{Lm:Sum mu}, it easily follows that
 \[
  \mu_{k,0} q^{k-1} - \delta' \sum_{t=1}^k \frac 1 t q^{k-2} \leq s_0(P) \leq \mu_{k,0} q^{k-1} + \left( 2 + \delta' \sum_{t=1}^k \frac 1 t \right) q^{k-2}.
 \]
Therefore, we can take
\[
 \delta_k = 2 + (\delta_{k-1} + 4) \sum_{t=1}^k \frac 1 t.
\]
This finishes the proof.
\end{proof}

We are now ready to prove \Cref{Thm:Main}.
We repeat the theorem statement for the reader's convenience.

\bigskip

\noindent \textbf{\Cref{Thm:Main}.}
\textit{For every integer $k \geq 3$, there exists a constant $\sigma_k$,  such that for all prime powers $q$, every intersecting family $\mf \subseteq \hom k$ with $|\mf| \geq \frac1{\sqrt 2} q^k + \sigma_k q^{k-1}$ is contained in a star.}

\begin{proof}
 By \Cref{Lm:Trivial but crucial}, it suffices to prove the theorem for intersecting families $\mf$ in $ERS(q,k)$.
 Let $\mq$ be the set of prime powers.
 We check that if we fix $k \geq 3$, the family of codes $(ERS(k,q))_{q \in \mq}$ satisfies the hypotheses of \Cref{Thm:Stability}.
Let us check this using the notation $\mm_q$, $\ms_q$, $a,b,\delta,\mu,\tau$ from \Cref{Thm:Stability}.
We do point out that the code $ERS(q,k)$ has dimension $k+1$, so we should replace $k$ by $k+1$ in the statement of \Cref{Thm:Stability}.
\begin{enumerate}
 \item Since $\ms_q$ is a normal rational curve in $\pg(k,q)$, no 3 points of $\ms_q$ are collinear by \cref{Res:Arc}.
 \item The length of the code is given by $n_q = q+1$, hence we can take $a = b = 1$.
 \item Write $\mu = \mu_{k,0}$ as defined in \Cref{Df:mu}. Then $|\mm_q|$ is a polynomial in $q$ with leading term $\mu q^k$, hence $\big| |\mm_q| - \mu q^k \big| \leq \tau q^{k-1}$ for some constant $\tau$.
 \item Let $\delta_k$ be as in \Cref{Prop:Bounded deviation}. Then by \Cref{Prop:Bounded deviation} and the previous point, for each point $P \notin \ms_q$,
 \begin{align*}
   \left| \sett{\Pi \in \mm_q}{P \in \Pi} - \frac{|\mm_q|}q \right| 
   & \leq \left| \sett{\Pi \in \mm_q}{P \in \Pi} - \mu_{k,0} q^{k-1} \right| + \left| \mu_{k,0} q^{k-1} - \frac{|\mm_q|}q \right| \\
   & \leq \delta_k q^{k-2} + \tau q^{k-2}.
 \end{align*}
 Hence, we can take $\delta = \delta_k + \tau$.
\end{enumerate}
Consider $\mu = \mu_{k,0} = \sum_{j=0}^k \frac{(-1)^j}{j!}$.
This is a partial sum of $\sum_{j=0}^\infty \frac{(-1)^j}{j!} = e^{-1}$.
Since it is an alternating series whose terms have decreasing absolute value, 
 \begin{align}
 \label{Eq:Mu bound}
  \frac 13 = \mu_{3,0} \leq \mu_{k,0} \leq \mu_{4,0} = \frac 38
 \end{align}
 for all $k \geq 3$.
This proves that $0 < \mu < 1$.
Moreover,
\[
 \frac{1}{\sqrt{1 + a^{-1}}} = \frac 1{\sqrt 2} > \frac 2 3 \geq 1 - \mu.
\]
Thus, \Cref{Thm:Stability} implies that the theorem statement holds.
\end{proof}

\begin{rmk}
 \label{Rmk:HM:3}
 One naturally wonders how close the value $\frac 1 {\sqrt 2}$ in \Cref{Thm:Stability} is to being optimal.
 As discussed before, an intersecting family of Hilton-Milner type would have size $(1-\mu_{k,0})q^k + \mo_q(q^{k-1})$, so it is not unreasonable to expect $1-\mu_{k,0}$ to be the optimal constant.
 For comparison, note that $\frac1{\sqrt 2} \approx 0.707$, and we have that $0.625 \leq 1-\mu_{k,0} \leq 0.666\dots$ by (\ref{Eq:Mu bound}).
 Moreover, $1-\mu_{k,0} \to 1 - e^{-1} \approx 0.632$ if $k \to \infty$.
\end{rmk}

\section{Conclusion}
 \label{Sec:Conclusion}

In this paper, we established a general framework for investigating intersection problems in linear codes.
We used it to prove a stability result concerning intersecting families of polynomials.
Some questions remain.

\begin{prob} 
 We proved that an intersecting family in $\hom k$ with $k \geq 3$ of size $\frac 1{\sqrt 2} q^k + \mo(q^{k-1})$ must be contained in a star?
 Can we improve the constant to $\sum_{j=1}^k \frac{(-1)^{j+1}}{j!}$?
 Can we prove that $\hom k$ satisfies the HM property?
\end{prob}

The constant $\sum_{j=1}^k \frac{(-1)^{j+1}}{j!}$ is minimal for $k = 2$, and a Hilton-Milner type family in $\hom 2$ would have size $\frac 12 (q^2+q)$.
It could hence be the case the establishing the HM property is hardest for $\hom 2$.
Moreover, there is an even more concrete difficulty.
If $\mf$ is a star or a Hilton-Milner type family, then there exists a special point incident with all or all but one block of $\mf$.
However, there are intersecting families of size $\ceil{\frac {q^2}2}$, with a vastly different structure, see \cite[Example 4.2]{Aguglia:Csajbok:Weiner} and \cite[Construction 4.4.3]{thesis}, where each point is incident with at most $q$ blocks of such a family.
An interesting open problem is to find similar constructions of higher degree.

\begin{prob}
 Can we construct large intersecting families $\mf$ in $\hom k$ where every point is incident with a small number of blocks of $\mf$?
\end{prob}

The characterisation of maximum $t$-intersecting families in $\hom k$ (and even in $\fq[X]_{\leq k}$) is still open for $t > 1$.

\begin{prob}
 We proved that if $t < k < q$, a $t$-intersecting family in $\hom k$ contains at most $q^{k+1-t}$ elements.
 Can we prove that only stars achieve equality?
 For certain ranges of the parameters $t$, $k$, $q$, this might be feasible using the \emph{spread approximations method} \cite{Kupavskii:Zakharov}.
 Moreover, can we improve the bound $\frac{q^2-1}k + 1$ on the size of $k$-intersecting families to a linear bound?
\end{prob}

Finally, it would be interesting to find other settings than linear codes where the techniques of this paper can yield stability results of intersecting families.
In \cite{A3-2022-Stability}, the author used the technique to prove stability results for intersecting families in \emph{circle geometries}.
As a special case, this yields a stability result of intersecting families in $\pgl(2,q)$.

\begin{prob}
 Given a permutation group $G$ acting on a set $X$, call $g, h \in G$ \emph{intersecting} if there exists an $x \in X$ with $x^g = x^h$.
 Use the techniques of the paper to prove stability results of intersecting families in certain classes of permutation groups.
\end{prob}

Of course, other settings than permutation groups would also be interesting.

\subsection*{Acknowledgements}

The author is thankful to Nathan Lindzey and Ferdinand Ihringer for interesting discussions concerning the contents of the paper.

\newcommand{\etalchar}[1]{$^{#1}$}

\appendix

\section{Tactical decompositions and translation schemes}
 \label{Sec:Assoc Sch}

In this section, we review the theory of association schemes, and in particular translation schemes.
We refer the reader to \cite[Chapter 2]{bcn} or \cite[Chapter 6]{Godsil} for more background.
Ultimately, we apply the theory determine the eigenvalues of certain graphs defined on the vector space of homogeneous polynomials over finite fields.

\subsection{Association schemes}

An association scheme admits several definitions, and can be seen as a partition $\ma = \{R_0, \dots, R_d\}$ of the cartesian product $X \times X$ of some finite set $X$ into some symmetric relations satisfying certain regularity conditions.
Each such relation induces a graph on $X$.
Let $A_i$ denote the adjacency matrix of $R_i$.

\begin{df}
 Given a finite set $X$, a partition $\ma = \{R_0, \dots, R_d\}$ of $X \times X$ is called a \emph{$d$-class (symmetric) association scheme} if $R_0$ is the identity relation (i.e.\ $A_0$ equals the identity matrix $I$), every adjacency matrix $A_i$ is symmetric, and the linear subspace of $\RR^{X \times X}$ generated by $A_0, \dots, A_d$ is a commutative algebra under the ordinary matrix product.
\end{df}


We denote by $\RR[\ma]$ the linear subspace of $\RR^{X \times X}$ spanned by $\ma$, called the \emph{Bose-Mesner algebra} of $\ma$.
A crucial property of association schemes is that $\RR[\ma]$ admits another important basis $E_0 = \frac1{|X|} J, E_1, \dots, E_d$, where each $E_j$ is an orthogonal projection, and the column spaces of the $E_j$'s constitute an orthogonal decomposition of $\RR^X$.
This basis is unique up to ordering.
We call it the \emph{dual basis} of $\RR[\ma]$.

The transition matrices between the two aforementioned bases of $\RR[\ma]$ capture crucial information about the association scheme.
Define the matrix $\BP \in \RR^{\{0, \dots, d\} \times \{0, \dots, d\}}$ by $A_j = \sum_{i=0}^d \BP(i,j) E_i$, and similarly the matrix $\BQ$ by $E_j = \frac 1{|X|} \sum_{i=0}^d \BQ(i,j) A_i$.
We call $\BP$ and $\BQ$ the \emph{matrices of eigenvalues} or \emph{character tables} of the association scheme.




\begin{df}
 Suppose that $(G,+)$ is a finite abelian group with identity element $0$, and that $S_0 = \set 0, S_1, \dots, S_d$ is a partition of $G$ into inverse-closed subsets.
 For each $S_i$, let $R_i$ be the relation corresponding to the Cayley graph $\Cay(G,S_i)$.
 If $\{R_0, \dots, R_d\}$ is an association scheme, it is called a \emph{translation scheme}.
\end{df}

Each translation scheme $\ma$ has a dual association scheme.
This works as follows.
The group $\hat G$ of characters of $G$ is isomorphic to $G$.
The characters of $G$ constitute an orthogonal basis of $\CC^G$ that diagonalises all the adjacency matrices $A_i$ of $\ma$.
Each matrix $E_j$ of the dual basis, when viewed as a complex matrix, is the orthogonal projection onto a subspace of $\CC^G$ that is spanned by some subset $T_j \subset \hat G$ of characters of $G$.
This yields a partition $T_0 = \set 0, T_1, \dots, T_d$ of the elements of $\hat G$, and this partition again gives rise to a translation scheme $\ma^*$ on $\hat G$.
Since $G \cong \hat G$, $\ma^*$ is isomorphic to a translation scheme on $G$.
If $\BP$ and $\BQ$ are the matrices of eigenvalues of $\ma$, then $\BQ$ and $\BP$ (in that order) are the matrices of eigenvalues of $\ma^*$.
Moreover, the dual translation scheme $\ma^{**}$ of $\ma^*$ is isomorphic to $\ma$.
If $\BP = \BQ$, then $\ma$ is said to be \emph{formally self-dual}.
If in addition $\ma$ is isomorphic to $\ma^*$, then $\ma$ is called \emph{self-dual}.

\subsection{Tactical decompositions}

Consider the projective space $\pg(k-1,q)$.
Let $\mathfrak S = \{S_1, \dots, S_d\}$ be a partition of the hyperplanes and $\mathfrak T = \{T_1, \dots, T_d\}$ a partition of the points.
Then $(\fs, \ft)$ is called a \emph{tactical decomposition} of $\pg(k-1,q)$ if there exist integers $s_{ij}$ and $t_{ij}$ such that
\begin{itemize}
  \item every hyperplane in $S_i$ is incident with exactly $s_{ij}$ points of $T_j$,
  \item every point in $T_i$ is incident with exactly $t_{ij}$ hyperplanes of $S_j$.
\end{itemize}
We note that the condition that $\fs$ and $\ft$ have the same number of parts (in our case $d$) is necessary for a tactical decomposition to exist, as proven by Block \cite[Corollary 2.1]{Block:67}.
For every group $G \leq \PGamL(k,q)$ of collineations, we can define $\fs$ and $\ft$ to be the sets of orbits of $G$ on the hyperplanes and points respectively.
Then $(\fs,\ft)$ is always a tactical decomposition.

Given a tactical decomposition $(\fs,\ft)$, define $S^*_0 = T^*_0 = \{\zero\}$, and for each $i=1,\dots,d$ define
\begin{align*}
 S^*_i &= \sett{ v \in \FF_q^k \setminus \set \zero }{ v^\perp \in S_i }, \\
 T^*_i &= \sett{ w \in \FF_q^k \setminus \set \zero }{ \vspan w\in T_i }.
\end{align*}
Then the Cayley graphs given by $\Cay(\FF_q^k, S_i)$ determine a translation scheme $\ma$ on $(\FF_q^k,+)$.
Likewise, the Cayley graphs given by $\Cay(\FF_q^k, T_i)$ determine a translation scheme $\ma^*$ on $(\FF_q^k,+)$, and these schemes are (isomorphic to) each others duals.

Using the the discussion of \Cref{Sec:Eigval Cayley} (or see \cite[Theorem 3.1.3]{thesis}), it is straightforward to determine that the matrices of eigenvalues of $\ma$ are given by
{\allowdisplaybreaks \begin{align*}
 \BP = \begin{pmatrix}
 1 & (q-1)|S_1| & \dots & (q-1)|S_d| \\
  1 & q t_{11} - |S_1| & \dots & q t_{1d} - |S_d| \\
  \vdots & & & \vdots \\
  1 & q t_{d1} - |S_1| & \dots & q t_{dd} - |S_d|
 \end{pmatrix},
  &&
 \BQ = \begin{pmatrix}
  1 & (q-1)|T_1| & \dots & (q-1)|T_d| \\
  1 & q s_{11} - |T_1| & \dots & q s_{1d} - |T_d| \\
  \vdots & & & \vdots \\
  1 & q s_{d1} - |T_1| & \dots & q s_{dd} - |T_d|
 \end{pmatrix}.
\end{align*}}
A \emph{polarity} of $\pg(k-1,q)$ is an involution of the subspaces of $\pg(k-1,q)$ that reverses inclusion.
A polarity maps points to hyperplanes and vice versa.
If $\zeta$ is a polarity that maps the set of hyperplanes $S_i$ to the set of points $T_i$ (and vice versa since $\zeta$ is a polarity) for each $i$, then $\zeta$ induces an isomorphism between $\ma$ and $\ma^*$.
Hence, in this case, $\ma$ is a self-dual translation scheme, and $\BP = \BQ$.

\subsection{Bivariate homogeneous polynomials of small degree}

Recall the map $\nu_k$ defined in \Cref{Df:nu}, and let $\mn$ denote its image, i.e.\ $\mn$ is the canonical normal rational curve in $\pg(k,q)$.
As discussed in \Cref{Res:Lift}, there is a subgroup $G$ of $\PGL(k+1,q)$ isomorphic to $\PGL(2,q)$ that stabilises $\mn$.
Note that the order of $\PGL(2,q)$ equals $q(q^2-1)$.
Since $\pg(k,q)$ contains more than $q^k$ points, the number of orbits of $G$ on the set of points is more than $q^{k-3}$.
In particular, if $k > 3$, then the number of orbits grows with $q$.
However, if $k \leq 3$, then it is known that the number of orbits is constant.
This yields translation schemes defined on $(\FF_q^{k+1},+)$, and in fact, these schemes can naturally be described as translation schemes on $(\hom k,+)$.

\subsubsection{Degree 2}

We define the following relations on $\hom 2$, where $f,g \in \hom 2$ are in relation
\begin{itemize}
 \item $R_0$ if $f = g$,
 \item $R_1$ if $f-g$ has a double linear factor,
 \item $R_2$ if $f-g$ has two distinct linear factors,
 \item $R_3$ if $f-g$ is irreducible.
\end{itemize}
This is a translation scheme.
The tactical decomposition consists of the point- and hyperplane orbits of the stabiliser of a normal rational $\mn$ curve in $\pg(2,q)$, see e.g.\ \cite[Chapter 8]{Hirschfeld:79}.
Let $S_1, S_2, S_3$ be the sets of lines that intersect $\mn$ in respectively 1, 2, or 0 points.
Let $T_1 = \mn$.

If $q$ is odd, we can partition the remaining points into two sets $T_2$ and $T_3$ depending on whether they are incident with 2 or 0 lines of $S_1$ respectively.
There exists a polarity that maps each $S_i$ to $T_i$, and vice versa, hence the translation scheme is self-dual.
We find that
\[
 \BP = \BQ = \begin{pmatrix}
  1 & q^2-1 & \frac 12 q(q^2-1) & \frac 12 q(q-1)^2 \\
  1 & -1 & \frac 12 q(q-1) & -\frac 12 q(q-1) \\
  1 & q-1 & -q & 0 \\
  1 & -q-1 & 0 & q 
 \end{pmatrix}.
\]

If $q$ is even, as discussed in \Cref{Ex:Nucleus}, there is a unique point $N$ such that $S_1$ is the set of lines through $N$.
We take $T_2 = \set N$, and $T_3$ the set of the remaining points.
Then
\begin{align*}
 \BP = \begin{pmatrix}
  1 & q^2-1 & \frac 12 q(q^2-1) & \frac 12 q(q-1)^2 \\
  1 & -1 & \frac 12 q(q-1) & -\frac 12 q(q-1) \\
  1 & q^2-1 & - \frac 12 q(q+1) & - \frac 12 q(q-1) \\
  1 & -1 & - \frac 12 q & \frac 12 q
 \end{pmatrix},
 &&
 \BQ = \begin{pmatrix}
  1 & q^2-1 & q-1 & (q+1)(q-1)^2 \\
  1 & -1 & q - 1 & -q+1 \\
  1 & q-1 & - 1 & -q+1 \\
  1 & -q-1 & -1 & q+1
 \end{pmatrix}.
\end{align*}

\subsubsection{Degree 3}

We define the relations on $\hom 3$.
We say that $f,g \in \hom 3$ are in relation
\begin{itemize}
 \item $R_0$ if $f=g$, 
 \item $R_1$ if $f-g$ has a triple linear factor,
 \item $R_2$ if $f-g$ has a double linear factor, and a distinct linear factor,
 \item $R_3$ if $f-g$ has three distinct linear factors,
 \item $R_4$ if $f-g$ is the product of a linear factor and an irreducible quadratic factor,
 \item $R_5$ if $f-g$ is irreducible over $\fq$.
\end{itemize}
This yields a translation scheme, corresponding to the point- and hyperplane orbits of the stabiliser of a normal rational curve in $\pg(3,q)$.
The numbers $s_{ij}$ and $t_{ij}$ can be found in \cite[Tables 1 and 2]{Bartoli:DMP:20}.

In case $5 \leq q \equiv \eps = \pm 1 \pmod 3$, there is a polarity that swaps the point- and hyperplane-orbits, and the scheme is self-dual, see \cite[Theorem 21.1.2 (iv)]{Hirschfeld:85}.
This yields the following matrices of eigenvalues:
\[
 \BP = \BQ = \begin{pmatrix}
  1 & q^2-1  & q(q^2-1) & \frac 16 (q-1)^2q(q+1) &  \frac 12 (q-1)^2q(q+1) &  \frac 13 (q-1)^2q(q+1) \\
  1 & -1 & q(q-1) & \frac 16 q(q-1)(2q-1) & -\frac 12 q(q-1) & -\frac 13 q (q-1)(q+1) \\
  1 & q-1 & q(q-2) & - \frac 12 q(q-1) & -\frac 12 q(q-1) & 0 \\
  1 & 2q-1 & -3q & \frac 16 q (\eps q + 5) & \frac 12 q(-\eps q + 1) & \frac 13 q (\eps q - 1) \\
  1 & -1 & -q & \frac 16 q (-\eps q + 1) & \frac 12 q(\eps q + 1) & \frac 13 q (-\eps q + 1) \\
  1 & -(q+1) & 0 & \frac 16 q (\eps q - 1) & \frac 12 q(-\eps q + 1) & \frac 13 q (\eps q + 2) \\
 \end{pmatrix}
\]

If $q > 3$ is a power of $3$, the scheme is not (formally) self-dual.
The matrices of eigenvalues are given by
\begin{align*}
 \BP &= \begin{pmatrix}
 1 & q^2-1  & q(q^2-1) & \frac 16 (q-1)^2q(q+1) &  \frac 12 (q-1)^2q(q+1) &  \frac 13 (q-1)^2q(q+1) \\
 1 & -1 & q(q-1) & \frac 16 q (q-1) (2q-1) & -\frac 12 q (q-1) & -\frac 13 (q-1)q(q+1) \\
 1 & q^2-1 & -q & -\frac 16 q (q-1) & -\frac 12 q (q-1) & -\frac 13 q (q-1) \\
 1 & -1 & q(q-1) & -\frac 16 q (3q-1) & -\frac 12 q (q-1) & \frac 13 q \\
 1 & -1 & -q & \frac 16 q (q+1) & -\frac 12 q (q-1) & \frac 13 q (q+1) \\
 1 & -1 & -q & -\frac 16 q (q-1) & \frac 12 q (q+1) & -\frac 13 q(q-1) \\
 \end{pmatrix} \\
 \BQ &= \begin{pmatrix}
 1 & q^2-1 & q^2-1 & (q-1)^2(q+1) & \frac 12 (q-1)^2 q (q+1) & \frac 12 (q-1)^2 q (q+1) \\
 1 & -1 & q^2-1 & -(q-1) & -\frac 12 q (q-1) & -\frac 12 q (q-1) \\
 1 & q-1 & -1 & (q-1)^2 & -\frac 12 q (q-1) & -\frac 12 q (q-1) \\
 1 & 2q-1 & -1 & -(3q-1) & \frac 12 q (q+1) & -\frac 12 q (q-1) \\
 1 & -1 & -1 & -(q-1) & -\frac 12 q (q-1) & \frac 12 q (q+1) \\
 1 & -(q+1) & -1 & 1 & \frac 12 q (q+1) & -\frac 12 q (q-1) \\
 \end{pmatrix}
\end{align*}

One application of the eigenvalues is that we can slightly improve the upper bound of \Cref{Prop:t int} (2) for 3-intersecting families in $\hom 3$ if $q > 3$ is a power of $3$.

\begin{prop}
 \label{Prop:3 int}
 Suppose that $\mf$ is a 3-intersecting family in $\hom 3$ with $q \geq 9$ a power of $3$.
 Then $|\mf| \leq \frac 13 q^2 - \frac 29 q$.
\end{prop}

\begin{proof}
 The Delsarte clique bound (see e.g.\ \cite[Corollary 3.7.2]{Godsil:Meagher}) says that if $R$ is a single relation in an association scheme, then a clique has it size bounded by $\frac{k}{|\tau|}+1$, if $R$ is $k$-regular, and its adjacency matrix has smallest eigenvalue $\tau$.
 For $3$-intersecting families in $\hom 3$ with $q > 3$ a power of 3, this yields the bound
 \[
  |\mf| \leq \floor{\frac{\frac 16 (q-1)^2 q(q+1)}{\frac 16 q(3q-1)} + 1}
  = \frac 13 q^2 - \frac 29 q. \qedhere
 \]
\end{proof}

\begin{rmk}
 If $q \geq 5$ is not a power of $3$, we can recover the bound $|\mf| \leq \frac{q^2-1}3 + 1$ from \Cref{Prop:t int} (2) (with only the weak inequality) using the linear programming approach of Delsarte \cite{Delsarte}.
\end{rmk}

\subsection{Trivariate homogeneous polynomials of degree 2}

We can also define a translation scheme on $\FF_q[X,Y,Z]_2$.
Say that $f,g \in \fq[X,Y,Z]_2$ are in relation
\begin{itemize}
 \item $R_0$ if $f=g$,
 \item $R_1$ if $f-g$ has a double linear factor,
 \item $R_2$ if $f-g$ has two distinct linear factors,
 \item $R_3$ if $f-g$ splits into two (conjugate) linear factors over $\FF_{q^2}$ (but not over $\fq$),
 \item $R_4$ if $f-g$ is irreducible over $\FF_{q^2}$.
\end{itemize}

The corresponding tactical decomposition is given by the point- and hyperplane-orbits of the stabiliser of the \emph{quadric Veronesean} in $\pg(5,q)$.
We refer the reader to \cite[Chapter 4]{Hirschfeld:Thas} for more details.
The parameters $s_{ij}$ and $t_{ij}$ were determined in \cite[Theorems 3.2 and 3.3]{AlnajjarineLavrauw}.

For $q$ odd, the translation scheme is self-dual because of the existence of a polarity that swaps the point- and hyperplane-orbits constituting the tactical decomposition, see e.g.\ \cite[Theorem 4.25]{Hirschfeld:Thas}.
We find that the matrices of eigenvalues $\BP = \BQ$ are given by:
\[
 \begin{pmatrix}
1 & q^3 - 1 & \frac12 q(q^2-1)(q^2+q+1) & \frac 12 q (q-1)^2 (q^2+q+1) & q^2 (q-1)^2 (q^2+q+1) \\
1 & -1 & \frac{1}{2} q (q+1)^2(q-1) & -\frac{1}{2} q(q-1)(q^2+1) & -q^2(q-1) \\
1 & q^2 - 1 & \frac{1}{2}q(q^2-2q-1) & \frac{1}{2}q (q-1)^2& -q^2(q-1) \\
1 & -(q^2 +1) & \frac{1}{2}q (q^2-1) & \frac{1}{2}q(q^2+1) & -q^2(q-1) \\
1 & -1 & -\frac{1}{2}q(q+1) & -\frac{1}{2}q(q-1) & q^2
 \end{pmatrix}.
\]

In case that $q$ is even, the scheme is no longer self-dual, and we find the following matrices of eigenvalues.

\begin{align*}
 \allowdisplaybreaks
\BP &= \begin{pmatrix}
1 & q^3 - 1 & \frac{1}{2} q(q^2-1)(q^2+q+1) &
\frac{1}{2} q(q-1)^2(q^2+q+1) & q^2(q-1)^2(q^2+q+1) \\
1 & -1 & \frac{1}{2} q(q+1)^2(q-1) &
-\frac{1}{2} q (q-1)(q^2+1) &
-q^2(q-1)\\
1 & q^3 - 1 & -\frac{1}{2}q(q+1) &
-\frac{1}{2}q(q-1) &
-q^2(q-1) \\
1 & -1 & \frac{1}{2}q(q^2-q-1) &
\frac{1}{2}q(q^2-q+1) &
-q^2(q-1) \\
1 & -1 & -\frac{1}{2}q(q+1) &
-\frac{1}{2}q(q-1) &
q^2
\end{pmatrix}, \\
\BQ &= \begin{pmatrix}
1 & q^3 - 1 & q^3 - 1 & (q+1)(q-1)^2(q^2+q+1) & q^2(q-1)^2(q^2+q+1) \\
1 & -1 & q^3 - 1 & -(q^2 - 1) & -q^2(q-1) \\
1 & q^2 - 1 & -1 & (q-1)(q^2-q-1) & -q^2(q-1) \\
1 & -(q^2 + 1) & -1 & q^3 + 1 & -q^2(q-1) \\
1 & -1 & -1 & -(q^2 - 1) & q^2
\end{pmatrix}.
\end{align*}

\begin{rmk}
 As explained in \cite{AlnajjarineLavrauw}, we can naturally represent the vector space underlying $\pg(5,q)$ as the space of symmetric $3 \times 3$ matrices over $\FF_q$.
 The point orbits consist of the projective points corresponding to rank 1 matrices, 2 orbits corresponding to rank 2 matrices, and an orbit corresponding to rank 3 matrices.
 As such, we can fuse relations 2 and 3 in the above dual translation schemes (both for $q$ odd and $q$ even) to obtain the \emph{symmetric bilinear forms} scheme over $\fq$, see e.g.\ \cite[\S 9.5 D]{bcn}.
\end{rmk}


\begin{thebibliography}{BDMP20}\setlength{\itemsep}{-2mm}

\bibitem[ACW24]{Aguglia:Csajbok:Weiner}
A.~Aguglia, B.~Csajb\'ok, and Zs. Weiner.
\newblock Intersecting families of graphs of functions over a finite field.
\newblock {\em Ars Math. Contemp.}, 24(1):Paper No. 4, 22, 2024.

\bibitem[Adr22a]{A3-2022-Stability}
S.~Adriaensen.
\newblock Stability of {E}rd{\H{o}}s-{K}o-{R}ado theorems in circle geometries.
\newblock {\em J. Combin. Des.}, 30(11):689--715, 2022.

\bibitem[Adr22b]{A-3-EKR}
S. Adriaensen.
\newblock Erd{\H o}s-{K}o-{R}ado theorems for ovoidal circle geometries and polynomials over finite fields.
\newblock {\em Linear Algebra Appl.}, 643:1--38, 2022.

\bibitem[Adr24]{thesis}
S.~Adriaensen.
\newblock {\em Stars and Weightlessness in Projective Space}.
\newblock PhD thesis, Vrije Universiteit Brussel, 2024.

\bibitem[AL25]{AlnajjarineLavrauw}
N.~Alnajjarine and M.~Lavrauw.
\newblock Webs and squabs of conics over finite fields.
\newblock {\em Finite Fields Appl.}, 102:Paper No. 102544, 33, 2025.

\bibitem[Bal20]{Ball}
S.~Ball.
\newblock {\em A course in algebraic error-correcting codes}.
\newblock Compact Textbooks in Mathematics. Birkh\"auser/Springer, Cham, 2020.

\bibitem[BB66]{Bose:Burton}
R.~C. Bose and R.~C. Burton.
\newblock A characterization of flat spaces in a finite geometry and the uniqueness of the {H}amming and the {M}ac{D}onald codes.
\newblock {\em J. Combinatorial Theory}, 1:96--104, 1966.

\bibitem[BB89]{Blokhuis:Bruen}
A.~Blokhuis and A.~A. Bruen.
\newblock The minimal number of lines intersected by a set of {$q+2$} points, blocking sets, and intersecting circles.
\newblock {\em J. Combin. Theory Ser. A}, 50(2):308--315, 1989.

\bibitem[BCN89]{bcn}
A.~E. Brouwer, A.~M. Cohen, and A.~Neumaier.
\newblock {\em Distance-regular graphs}, volume~18 of {\em Ergebnisse der Mathematik und ihrer Grenzgebiete (3) [Results in Mathematics and Related Areas (3)]}.
\newblock Springer-Verlag, Berlin, 1989.

\bibitem[BDMP20]{Bartoli:DMP:20}
D.~Bartoli, A.~A. Davydov, S.~Marcugini, and F.~Pambianco.
\newblock On planes through points off the twisted cubic in {${\rm PG}(3,q)$} and multiple covering codes.
\newblock {\em Finite Fields Appl.}, 67:101710, 25, 2020.

\bibitem[Blo67]{Block:67}
R.~E. Block.
\newblock On the orbits of collineation groups.
\newblock {\em Math. Z.}, 96:33--49, 1967.

\bibitem[BW25]{Bulavka:Woodroofe}
D.~Bulavka and R.~Woodroofe.
\newblock Strict erd$\backslash$h $\{$o$\}$ s-ko-rado theorems for simplicial complexes.
\newblock {\em arXiv preprint arXiv:2503.15608}, 2025.

\bibitem[Del73]{Delsarte}
Ph. Delsarte.
\newblock {\em An algebraic approach to the association schemes of coding theory}.
\newblock PhD thesis, Philips Research Reports. Supplements, 1973.

\bibitem[EKR61]{EKR}
P.~Erd\H{o}s, C.~Ko, and R.~Rado.
\newblock Intersection theorems for systems of finite sets.
\newblock {\em Quart. J. Math. Oxford Ser. (2)}, 12:313--320, 1961.

\bibitem[Ell22]{Ellis}
D.~Ellis.
\newblock Intersection problems in extremal combinatorics: theorems, techniques and questions old and new.
\newblock {\em Surveys in combinatorics}, pages 115--173, 2022.

\bibitem[FHI{\etalchar{+}}25]{Trees}
P.~Frankl, G.~Hurlbert, F.~Ihringer, A.~Kupavskii, N.~Lindzey, K.~Meagher, and V.~R.~T. Pantangi.
\newblock Intersecting families of spanning trees.
\newblock {\em arXiv preprint arXiv:2502.08128}, 2025.

\bibitem[GM16]{Godsil:Meagher}
C.~Godsil and K.~Meagher.
\newblock {\em Erd\H os-{K}o-{R}ado theorems: algebraic approaches}, volume 149 of {\em Cambridge Studies in Advanced Mathematics}.
\newblock Cambridge University Press, Cambridge, 2016.

\bibitem[God10]{Godsil}
C.~Godsil.
\newblock Association schemes, 2010.
\newblock \url{https://www.math.uwaterloo.ca/~cgodsil/pdfs/assoc2.pdf}.

\bibitem[Hae95]{Haemers}
W.~H. Haemers.
\newblock Interlacing eigenvalues and graphs.
\newblock {\em Linear Algebra Appl.}, 226/228:593--616, 1995.

\bibitem[Hir79]{Hirschfeld:79}
J.~W.~P. Hirschfeld.
\newblock {\em Projective geometries over finite fields}.
\newblock Oxford Mathematical Monographs. The Clarendon Press, Oxford University Press, New York, 1979.

\bibitem[Hir85]{Hirschfeld:85}
J.~W.~P. Hirschfeld.
\newblock {\em Finite projective spaces of three dimensions}.
\newblock Oxford Mathematical Monographs. The Clarendon Press, Oxford University Press, New York, 1985.
\newblock Oxford Science Publications.

\bibitem[HM67]{Hilton:Milner}
A.~J.~W. Hilton and E.~C. Milner.
\newblock Some intersection theorems for systems of finite sets.
\newblock {\em Quart. J. Math. Oxford Ser. (2)}, 18:369--384, 1967.

\bibitem[HT16]{Hirschfeld:Thas}
J.~W.~P. Hirschfeld and J.~A. Thas.
\newblock {\em General {G}alois geometries}.
\newblock Springer Monographs in Mathematics. Springer, London, 2016.

\bibitem[KZ24]{Kupavskii:Zakharov}
A.~Kupavskii and D.~Zakharov.
\newblock Spread approximations for forbidden intersections problems.
\newblock {\em Adv. Math.}, 445:Paper No. 109653, 29, 2024.

\bibitem[LZ22]{Liu=Zhou}
X.~Liu and S.~Zhou.
\newblock Eigenvalues of {C}ayley graphs.
\newblock {\em Electron. J. Combin.}, 29(2):Paper No. 2.9, 164, 2022.

\bibitem[Mea19]{Maegher}
K.~Meagher.
\newblock An {E}rd{\H o}s-{K}o-{R}ado theorem for the group {${\rm PSU}(3,q)$}.
\newblock {\em Des. Codes Cryptogr.}, 87(4):717--744, 2019.

\bibitem[MS77]{MacWilliams:Sloane}
F.~J. MacWilliams and N.~J.~A. Sloane.
\newblock {\em The theory of error-correcting codes. {I}}, volume Vol. 16 of {\em North-Holland Mathematical Library}.
\newblock North-Holland Publishing Co., Amsterdam-New York-Oxford, 1977.

\bibitem[ST25]{Salia:Toth}
N.~Salia and D.~T{\'o}th.
\newblock Intersecting families of polynomials over finite fields.
\newblock {\em Finite Fields Appl.}, 101:Paper No. 102540, 14, 2025.

\bibitem[TT25]{Tanaka:Tokushige}
H.~Tanaka and N.~Tokushige.
\newblock Extremal problems for intersecting families of subspaces with a measure.
\newblock {\em European J. Combin.}, 127:Paper No. 104156, 18, 2025.

\end{thebibliography}
\end{document}